\documentclass[12pt]{article}
\usepackage{etex}
\usepackage{subfigure}
\usepackage{amsmath,amssymb,amsfonts,amsthm}
\usepackage{a4wide}
\usepackage{color}
\usepackage{verbatim}
\usepackage{epsfig,subfigure,epstopdf}
\usepackage[]{graphics}
\usepackage{graphicx,inputenc}
\usepackage{subfigure}
\usepackage[justification=centering]{caption}
\usepackage{pictex}
\usepackage{wrapfig}
\usepackage{multirow}
\usepackage{float}
\usepackage{setspace}
\usepackage[T1]{fontenc}
\usepackage{authblk}
\usepackage[textwidth=6.0in,textheight=9in]{geometry}

\usepackage[colorlinks,linktocpage,linkcolor=blue]{hyperref}
\usepackage{fancyhdr}

\newtheoremstyle{exampstyle}
  {\topsep} 
  {\topsep} 
  {\itshape} 
  {} 
  {\bfseries} 
  {.} 
  {.5em} 
  {} 

\theoremstyle{exampstyle}

\newtheorem{theorem}{Theorem}
\newtheorem{lemma}{Lemma}
\newtheorem{assumption}{Assumption}
\newtheorem{remark}{Remark}
\newtheorem{definition}{Definition}

\newtheorem{proposition}[theorem]{Proposition}

\allowdisplaybreaks

\usepackage{parskip}
\let\oldref\ref
\renewcommand{\ref}[1]{(\oldref{#1})}  
\renewcommand{\eqref}[1]{(\oldref{#1})}

\newbox\boxaddrone \newbox\boxaddrtwo

\def\D{\partial_t^\alpha}
\def\N+{n\in\mathbb{N}^{+}}

\def\A{\mathcal{A}}
\def\l{\langle}
\def\rd{\rangle_{L^2(D)}}

\def\V{\mathbb{V}}
\def\E{\mathbb{E}}

\def\A{\mathcal{A}}

\def\I{I_t^\alpha}
\def\o{\omega}

\def\W{\mathbb{W}}
\def\Vc{\mathcal{V}}

\begin{document}

\title{\large\textbf{ Recovery of the time-dependent source term in the 
stochastic fractional diffusion equation with heterogeneous medium}}
\author[1]{Shubin Fu\thanks{shubinfu89@gmail.com}}
\author[2]{Zhidong Zhang\thanks{zhidong.zhang@helsinki.fi}}
\affil[1]{\normalsize{Department of Mathematics, University of Wisconsin-Madison, USA}}
\affil[2]{\normalsize{Department of Mathematics and Statistics, University of Helsinki, Finland}}

\maketitle

\begin{abstract}
\noindent In this work, an inverse problem in the fractional diffusion equation  
with random source is considered. The measurements used are the statistical moments of the realizations of single point data 
$u(x_0,t,\omega).$ We build the representation of the solution $u$ in 
integral sense, then prove that the unknowns can be bounded by the moments theoretically. For the numerical reconstruction, we establish an iterative algorithm with regularized Levenberg-Marquardt type and some numerical results generated from this algorithm are displayed.    
For the case of highly heterogeneous media, the Generalized Multiscale finite element method (GMsFEM) will be employed. \\


\end{abstract}


\section{Introduction}

\subsection{Mathematical statement} 
The mathematical model in this work is stated as follows:  
 \begin{equation}\label{SDE}
  \begin{cases}
   \begin{aligned}
    \D u+\A u&=f(x)[g_1(t)+g_2(t)\dot{\W}(t)]=:f(x)g(t,\omega), 
    &&(x,t)\in D\times(0,T],\\
    u(x,t)&=0, &&(x,t)\in \partial D\times(0,T],\\
    u(x,0)&=0, &&x\in D.
   \end{aligned}
  \end{cases}
 \end{equation}
The domain $D\subset\mathbb{R}^d,\ d=1,2,3$ has sufficiently smooth boundary, 
and $\D$ with $\alpha\in (1/2,1)$ denotes the Djrbashyan-Caputo fractional derivative, defined as  
\begin{equation*}
	\D \psi(t) = \frac{1}{\Gamma(1-\alpha)}
	\int_0^t(t-\tau)^{-\alpha}\psi'(\tau)\ d\tau,
\end{equation*}
where $\Gamma(\cdot)$ is the gamma function. 
The lower bound $\alpha>1/2$ is set to ensure the well definedness of 
the Ito integral $\I g(t,\omega)$ and this can be seen in the next section. The operator $\A:H^2(D)\mapsto L^2(D)$ 
is an elliptic operator defined as $\A\psi(x)=-\triangledown\cdot(\kappa(x)\triangledown 
\psi(x))$, and $\kappa(x)$ may be highly heterogeneous. The source term $f(x)g(t,\omega)=f(x)[g_1(t)+g_2(t)\dot{\W}(t)]$ 
contains the targeted unknowns $g_1,g_2$, while the spatial component $f(x)$ is given. 
$\dot{\W}(t)$ is the white noise derived from the Brownian motion and then 
$g(t,\omega)=g_1(t)+g_2(t)\dot{\W}(t)$ constitutes an Ito process, 
see \cite{Bernt2003stochastic} for details. 

Our data is the moments of the 
realizations of $u$ on a single point $x_0\in D$ with the restriction $x_0\notin \text{supp} (f)$, which is different from \cite{LiuWenZhang:2019}. This condition will make the inverse problem more challenging in mathematics, but is meaningful in practical application. For instance, regarding equation \eqref{SDE} as the contaminant diffusion system, solution $u$ will be the concentration of pollutant, and $\text{supp}(f)$ is the location of pollution source, in which it is severely polluted. If the pollutant is harmful for human body, it is not allowed to observe inside the support of $f$  considering staff's health. Reflected on mathematics, we should set the restriction $x_0\notin \text{supp}(f)$, even though it will increase the difficulty. Actually, given $x_0\in \text{supp} (f)$, this inverse problem will be reduced to Volterra integral equations,  and the stability result and iterative algorithm follow straightforwardly. See Remark \ref{x0insuppf} for details.

The precise mathematical description of this inverse problem is given as follows: observe $u(x_0,t,\omega),\ x_0\notin \text{supp} (f)$,  
then use the statistical moments of the measurements to reconstruct $g_1,g_2$ simultaneously.

\subsection{Physical background and literature} 
In microscopic level, the random motion of a single particle can be 
viewed as a diffusion process. The classical diffusion equation 
can be deduced to describe the motion of particles, if we assume the 
key condition, the mean squared displacement of jumps after a long time is proportional to time, i.e. $\overline{(\Delta x)^2}\propto t,\ t\to \infty$. However, recently, people found some anomalous diffusion phenomena 
\cite{BarkaiMetzlerKlafter:2000,BouchaudGeorges:1990,
GefenAharonyAlexander:1983,KlafterSilbey:1980}, in which  
the assumption $\overline{(\Delta x)^2}\propto t,\ t\to \infty$ is 
violated. Sometimes it may possess the asymptotic behavior of $t^\alpha$, 
i.e. $\overline{(\Delta x)^2}\propto t^\alpha,\ \alpha\ne1.$ 
The different rate will lead to a reformulation to the diffusion equation, 
introducing the time fractional derivative in it, and the corresponding 
equations are called fractional differential equations (FDEs). 
We list some applications of FDEs, to name a few,
the thermal diffusion in media with fractal geometry \cite{Nigmatullin:1986},  
ion transport in column experiments \cite{hatano1998dispersive}, 
dispersion in a heterogeneous aquifer \cite{adams1992field}, non-Fickian diffusion in geological 
formations \cite{berkowitz2006modeling}, the analysis on viscoelasticity in material science 
\cite{mainardi2010fractional,wharmby2013generalization}. 
\cite{MetzlerJeonCherstvyBarka:2014} provides an extensive list.

If uncertainty is added in the source term, the FDE system will become 
more complicated and meaningful. Since it is common to meet a diffusion 
source, which is defined as a stochastic process to describe the uncertain 
character imposed by nature. As a consequence, it is worth to investigate 
the diffusion system with a random source. In such situation the solution 
$u$ will be written as a stochastic process, which makes the analysis more  challenging.  

In addition, to deal with the case of highly heterogeneous medium $\kappa(x)$, the Generalized Multiscale Finite Element Method (GMsFEM \cite{gmsfem}) will be used to simulate the forward problem of equation \eqref{SDE}. The introduction of GMsFEM will be given in section \ref{gmsfem}. 

For a comprehensive understanding of fractional calculus and FDEs, 
see \cite{KilbasSrivastavaTrujillo:2006,SamkoKilbasMarichev:1993, 
baleanu2019handbook} and the references therein. 
For inverse problems in FDEs, \cite{JinRundell:2015} is an extensive review. 
See \cite{FengLiWang:2019, Zhang:2017, NiuHelinZhang:2019} for inverse source 
and coefficient problems; see \cite{JiangLiLiuYamamoto:2017} for unique 
continuation principle; see \cite{HuangLiYamamoto:2019} for Carleman 
estimate in FDEs; see \cite{GhoshRulandSaloUhlmann:2020, LaiLinRuland:2019, RulandSalo:2020} 
for fractional Calderon problem. 
Furthermore, if we extend the assumption 
$\overline{(\Delta x)^2}\propto t^\alpha$ to a more general case 
$\overline{(\Delta x)^2}\propto F(t)$, the multi-term fractional 
diffusion equations and even the distributed-order differential equations will be generated, \cite{RundellZhang:2017, SunLiu:2020, 
LiLiuYamamoto:2015,LiKianSoccorsi:2019}. For numerical methods for inverse problems, see \cite{BaoYinZeng:2019, BaoChenLi:2017} and the references therein.  
Literature about the GMsFEM and its applications can be found in \cite{gmsfem,chung2014adaptive,chan2016adaptive,chung2015mixed,cho2017frequency}.

\subsection{Main result and outline}
Throughout this paper, the following restrictions on spatial component $f$, observation point $x_0$, and the unknowns $g_1,g_2$ 
are supposed to be valid.
\begin{assumption}\label{assumption}
\hfill
 \begin{itemize}
  \item $g_l\in L^\infty(0,T),\ l=1,2$, and set $M>0$ such that $\|g_1\|_{L^\infty(0,T)}\le M<\infty$; 
  \item $g_1$ changes its sign $N$ times on $(0,T)$ and $N<\infty$;
  \item $f\in H^2(D)\cap H_0^1(D)$ and $0\le f(x)\le C_f<\infty$ for $x\in D$;
  \item $x_0\notin \text{supp}(f)$, namely, $f(x_0)=0$.
 \end{itemize}
\end{assumption}

Now we can state the main result, which says the unknowns can be limited by some statistical moments of observations $u(x_0,t,\o)$.  
\begin{theorem}\label{bound}
 Under Assumption \ref{assumption}, let $v(x,t)$ satisfy equation \eqref{v} and define
 \begin{equation*}
  C_\alpha=1/\Gamma(2-\alpha),\ B_\eta= \| v(x_0,\cdot)\|^{-1}_{L^1(0,\eta)},\ \eta>0\ \text{be small}.
 \end{equation*}
 Then the following estimates for $g_1$ and $g_2$ are valid. 
 \begin{itemize}
  \item[(a)] If $N=0$,
\begin{equation*}
 \|g_1\|_{L^1(0,T-\eta)}\le C_\alpha B_\eta T^{1-\alpha}\ \E\big[\|u(x_0,\cdot,\o)\|_{L^1(0,T)}\big].
\end{equation*}
  \item[(b)] If $N>0$,
\begin{equation*}
 \|g_1\|_{L^1(0,T-\eta)}\le\frac{(B_\eta C_fT+1)^{N+1}-1}{B_\eta C_fT}
 \Big(C_\alpha B_\eta T^{1-\alpha}\ \E\big[\|u(x_0,\cdot,\o)\|_{L^1(0,T)}\big]+2M\eta\Big).
\end{equation*}
  \item[(c)] \begin{equation*}
 \|g_2\|_{L^2(0,T-\eta)}\le \eta^{1/2} B_\eta
 \ \big\|\V[I_t^{1-\alpha} u(x_0,t,\omega)]\big\|_{L^1(0,T)}^{1/2}.
\end{equation*}
 \end{itemize}
\end{theorem} 
In this theorem, $I_t^{1-\alpha}$ means the fractional 
integral operator, and $\E,\V$ are the notations for expectation and 
variance, respectively. These knowledge can be seen in section 
\ref{section:preliminary}. Noting that the stochastic process $g_2(t)\dot{\W}(t)$ is independent of the sign of $g_2$ by the properties of $\W$ in section \ref{section:preliminary}, sequentially we consider $g_2^2$ instead of $g_2$. That's why the $L^2$ norm of $g_2$ is estimated.   

The remaining part of this manuscript is structured as follows.   Section \ref{section:preliminary} includes the preliminaries, such as the probability space $(\Omega,\mathcal{F},P)$ and the stochastic solution $u$ for equation \eqref{SDE}. Also, some auxiliary results like the reverse convolution inequality and the 
maximum principles in fractional diffusion equations are collected.  
In section \ref{section:main}, we prove Theorem \ref{bound}. After that the numerical reconstruction for the unknowns is investigated in section \ref{section:numerical}. We construct the regularized Levenberg-Marquardt iteration \eqref{iteration},   
and prove its convergence--Proposition \ref{convergence}. Some numerical results generated by iteration \eqref{iteration} are also displayed. Furthermore, some brief knowledge of GMsFEM is provided in this section.

\section{Preliminaries}\label{section:preliminary}
\subsection{Brownian motion and Ito isometry formula}
To state the Ito formula, firstly we need to give the setting of 
probability space.  
\begin{definition}
 We call $(\Omega,\mathcal{F},P)$ a probability space 
 if $\Omega$ denotes the nonempty sample space, 
 $\mathcal{F}$ is the $\sigma-$algebra of $\Omega$ and 
 $P:\mathcal{F}\mapsto [0,1]$ is the probability measure. 
\end{definition}
With the above definition, the expectation $\E$ and variance $\V$ of 
a random variable $X$ can be given as   
 \begin{equation*}
  \E[X]=\int_\Omega X(\o)\ dP(\o),\ 
  \V[X]=\E[(X-E[X])^2].
 \end{equation*}

The Brownian motion $\W(t)$, which is also called Wiener process in mathematics, has the following properties, 
 \begin{itemize}
  \item $\W(0)=0;$
  \item $\W(t)$ has continuous paths;  
  \item $\W(t)$ has independent increments and satisfies 
  $$\W(t)-\W(s)\sim \mathcal{N}(0,t-s),\ 0\le s\le t,$$
  where $\mathcal{N}$ is the normal distribution. 
 \end{itemize}

Now the essential tool, Ito isometry formula can be stated. 
\begin{lemma}{(\cite{Bernt2003stochastic}).}\label{Ito isometry formula} 
Let $(\Omega,\mathcal{F},P)$ be a probability space and  
$\psi: [0,\infty)\times \Omega\rightarrow\mathbb{R}$
satisfy the following properties.
\begin{itemize}
 \item[(1)] $(t,\omega)\rightarrow \psi(t,\omega)$ is 
 $\mathcal{B}\times\mathcal{F}$-measurable, where $\mathcal{B}$ denotes the Borel $\sigma$-algebra on  $[0,\infty);$
 \item[(2)] $\psi(t,\omega)$ is $\mathcal{F}_t$-adapted;
 \item[(3)] $\E [\int_0^S\psi^2(t,\omega)\ dt] <\infty$ for some $S>0$.
\end{itemize}
Then the Ito integral $\int_0^S \psi(t,\o)\ d\W(t)$, 
where $d\W(t)$ denotes the random measure derived from $\W$, 
is well defined, and it follows that
\begin{equation*}
 \E\Big[\Big(\int_0^S\psi(t,\omega)\ d\W(t)\Big)^2\Big]
=\E\Big[\int_0^S\psi^2(t,\omega)\ dt\Big].
\end{equation*}
\end{lemma}

\subsection{Stochastic weak solution}
The randomness from $\dot{\W}(t)$ means that we can not differentiate $u$ in $t$ for each $\o\in\Omega$. As a consequence, we will define the weak solution $u$ of equation \eqref{SDE} in the integral sense.  

Firstly, the fractional integral operator $\I$ and the corresponding Ito integral $\I g(t,\omega)$ are given. 
\begin{definition}
 The fractional integral operator $\I,\ \alpha\in(1/2,1)$ is defined as
\begin{equation*}
 \I \psi(t)=\Gamma(\alpha)^{-1}\int_0^t(t-\tau)^{\alpha-1}\psi(\tau)\ d\tau, \quad t>0. 
\end{equation*}
Then we define $\I g(t,\omega)$ as  
 \begin{equation*}
  \I g(t,\omega)=\I g_1(t)+\Gamma(\alpha)^{-1}\int_0^t(t-\tau)^{\alpha-1}g_2(\tau)\ d\W(\tau).
 \end{equation*}
\end{definition}

Now we explain the necessity of the restriction $\alpha\in(1/2,1)$.  For $t\in(0,\infty)$, 
from the conditions $\alpha\in(1/2,1)$ and $\|g_2\|_{C(0,\infty)}\le M$, we have $(t-\tau)^{\alpha-1}g_2(\tau)$ is square-integrable 
on $(0,t).$ Then Lemma \ref{Ito isometry formula} yields that the Ito integral 
$\int_0^t(t-\tau)^{\alpha-1}g_2(\tau)\ d\W(\tau)$ is well defined. 

In addition, the direct calculation gives that 
\begin{equation*}
 \I \D \psi (t)=\psi(t)-\psi(0),
\end{equation*}
which implies the next definition of the weak solution for 
equation \eqref{SDE}. 
\begin{definition}[Stochastic weak solution]\label{weak solution}
The stochastic process 
$u(\cdot,t,\omega):(0,T]\times\Omega\mapsto L^2(D)$ is called as  
a stochastic weak solution of equation \eqref{SDE} if for each 
$\psi\in H^2(D)\cap H_0^1(D)$ and $\omega\in \Omega$, it holds that   
 \begin{equation*}
  \l u(\cdot,t,\o), \psi(\cdot)\rd+\l \I\A u(\cdot,t,\o), \psi(\cdot)\rd = \I g(t,\o)\ \l f(\cdot), \psi(\cdot)\rd,\ t\in (0,T].
 \end{equation*}
\end{definition}

\subsection{Auxiliary lemmas}
Here we list some auxiliary lemmas which will be used later. First 
the reverse convolution inequality is given. 
\begin{lemma}{(\cite[Lemma 3.1]{LiuZhang:2017}).}\label{lemma_reverse_convolution}
Let $0\le T_1<T_2<\infty$ and $\eta>0$ be arbitrarily given.
Suppose that $\varphi_1\in L^1(T_1,T_2+\eta)$, $\varphi_2\in L^1(0,T_2-T_1+\eta)$
and $\varphi_2\ge0$ on $(0,T_2-T_1+\eta)$.

\begin{itemize}
\item [(a)] If $\varphi_1$ keeps its sign on $(T_1,T_2+\eta)$, then
\begin{equation}\label{a}
\|\varphi_1\|_{L^1(T_1,T_2)}\|\varphi_2\|_{L^1(0,\eta)}
\le\Big\|\int_{T_1}^t\varphi_1(s)\varphi_2(t-s)\ ds\Big\|_{L^1(T_1,T_2+\eta)}.
\end{equation}
\item [(b)] If $\varphi_1$ only keeps its sign on $(T_1,T_2)$, then
\begin{equation}\label{b}
\begin{aligned}
\|\varphi_1\|_{L^1(T_1,T_2)}\|\varphi_2\|_{L^1(0,\eta)}
\le&\Big\|\int_{T_1}^t\varphi_1(s)\varphi_2(t-s)\ ds\Big\|_{L^1(T_1,T_2+\eta)}\\
&+2\|\varphi_1\|_{L^1(T_2,T_2+\eta)}\|\varphi_2\|_{L^1(0,\eta)}.
\end{aligned}
\end{equation}
\end{itemize}
\end{lemma}

The next lemmas are the maximum principles in FDEs. 
\begin{lemma}{(Maximum principle, \cite[Theorem 2]{Luchko:2009}).}\label{max}
Fix $T\in (0,\infty),$ let $\psi$ satisfy the following 
fractional diffusion equation 
 \begin{equation}\label{FDE}
    \D \psi+\A \psi=F(x,t),\ (x,t)\in D\times(0,T),
 \end{equation}
 and define $\lambda_T=\partial D\times [0,T]\cup \overline{D}\times \{0\}$. If $F\le 0$, then 
\begin{equation*}
 \psi(x,t)\le \max\Big\{0,\max\{\psi(x,t):(x,t)\in \lambda_T\}\Big\},
 \ (x,t)\in \overline D\times [0,T].
\end{equation*}
\end{lemma}

\begin{lemma}{(Strong maximum principle, \cite[Theorem 1.1]{LiuRundellYamamoto:2016}).}\label{strong_max}
 We set $\psi$ as the solution of equation \eqref{FDE} and let 
 $F=0,\ \psi(x,0)\ge 0,\ \psi(x,0)\not\equiv0,$ $\psi(x,t)=0$ on   $\partial D\times(0,T)$. Then for any $x_0\in D$, 
 the set $\{t>0:\psi(x_0,t)\le 0\}$ is at most a finite set.  
\end{lemma}

In addition, we gives a representation lemma for the weak solution 
$u(x,t,\o)$.
\begin{lemma}{(\cite[Lemma 5]{LiuWenZhang:2019})}\label{lem:uv}
 The weak solution $u$ of equation \eqref{SDE} can be written as
 \begin{equation}\label{duhamel}
  \begin{aligned}
   u(x,t,\o)=\I g(t,\o) f(x)+\int_0^t \I g(\tau,\o) v_t(x,t-\tau)\ d\tau,\quad t\in(0,T],
  \end{aligned}
 \end{equation}
 where $v(x,t)$ is the solution of the following deterministic
 fractional diffusion equation
 \begin{equation}\label{v}
  \begin{cases}
   \begin{aligned}
    \D v+\A v &=0,&& (x,t)\in D\times(0,T],\\
    v(x,t)&=0,&& (x,t)\in \partial D\times(0,T],\\
    v(x,0)&=f(x),&& x\in D.
   \end{aligned}
  \end{cases}
 \end{equation}
\end{lemma}

\section{Main result}\label{section:main}
\subsection{Moments representation}
With Lemmas \ref{Ito isometry formula} and \ref{lem:uv}, the moments we used can be represented in terms of the unknowns $g_1,g_2$. See the lemma below. 
\begin{lemma}
 \begin{equation}\label{moment_unknown}
\begin{aligned}
 \E[ I^{1-\alpha}_t u(x_0,t,\o)]
 &=\int_0^t g_1(\tau)  v(x_0,t-\tau)\ d\tau,\\
 \V[ I^{1-\alpha}_t u(x_0,t,\o)]
 &=\int_0^t g_2^2(\tau) [ v(x_0,t-\tau)]^2\ d\tau,\quad t\in(0,T].
 \end{aligned}
\end{equation}
\end{lemma}
\begin{proof}
With Lemma \ref{lem:uv} and the fact that 
\begin{equation*}
 \int_s^t (t-\tau)^{-\alpha}(\tau-s)^{\alpha-1} \ d\tau
 =B(1-\alpha,\alpha)=\Gamma(1-\alpha)\Gamma(\alpha)/\Gamma(1),
\end{equation*}
here $B$ is the Beta function, the next result can be deduced, 
\begin{align*}
 I^{1-\alpha}_t u(x,t,\o)
 =& \frac{f(x)}{\Gamma(\alpha)\Gamma(1-\alpha)}\int_0^t 
 (t-\tau)^{-\alpha}\int_0^\tau (\tau-s)^{\alpha-1} g(s,\o)\ ds\ d\tau\\
 &+ \frac{1}{\Gamma(1-\alpha)}\int_0^t (t-\tau)^{-\alpha}\int_0^\tau 
 \I g(\tau-s,\o) v_t(x,s)\ ds\ d\tau \\
 =& \frac{1}{\Gamma(\alpha)\Gamma(1-\alpha)}\Big[ f(x)\int_0^t g(s,\o)
 \int_s^t (t-\tau)^{-\alpha}(\tau-s)^{\alpha-1} \ d\tau\ ds\\
 &+\int_0^t v_t(x,s)\int_0^{t-s} g(r,\o) \int_{r+s}^t(t-\tau)^{-\alpha} (\tau-s-r)^{1-\alpha} \ d\tau\ dr\ ds\Big]\\
 =&f(x)\int_0^t g(s,\o)\ ds
 +\int_0^t v_t(x,s)\int_0^{t-s} g(r,\o) \ dr\ ds\\
 =&f(x)\int_0^t g(s,\o)\ ds
 +\int_0^t g(r,\o) [v(x,t-r)-v(x,0)]\ dr\\
 =&\int_0^t g(\tau,\o) v(x,t-\tau)\ d\tau.
 \end{align*}
Then we have 
\begin{equation*}\label{equality_2}
 I^{1-\alpha}_t u(x_0,t,\o)=\int_0^t g_1(\tau) v(x_0,t-\tau)\ d\tau
 +\int_0^t g_2(\tau) v(x_0,t-\tau)\ d\W(\tau).
\end{equation*} 
Applying Ito formula in Lemma \ref{Ito isometry formula} to 
the above equality leads to \eqref{moment_unknown}.
\end{proof}

\begin{remark}\label{x0insuppf}
In \cite{LiuWenZhang:2019}, the authors use integration by parts on the right side of \eqref{moment_unknown} to deduce the following second kind Volterra equations, 
\begin{equation*}
 \begin{aligned}
  G_1(t)&=f^{-1}(x_0)\E[ I^{1-\alpha}_t u(x_0,t,\o)]-f^{-1}(x_0)\int_0^t G_1(\tau)v_t(x_0,t-\tau)\ d\tau,\\
  G_2(t)&=f^{-2}(x_0)\V[ I^{1-\alpha}_t u(x_0,t,\o)]-2f^{-2}(x_0)\int_0^t G_2(\tau) v(x_0,t-\tau)
   v_t(x_0,t-\tau)\ d\tau,
 \end{aligned}
\end{equation*}
where $$G_1(t)=\int_0^t g_1(\tau)\ d\tau,\quad  G_2(t)=\int_0^t g^2_2(\tau)\ d\tau.$$ 
\end{remark}
However, since $f(x_0)=0$ in this work, we can only start the analysis from \eqref{moment_unknown}. Due to the convolution structure, the estimates of the unknowns on the partial interval $(0,T-\eta)$ are attained. See the next subsection for details.

\subsection{Proof of Theorem \ref{bound}}
From \eqref{moment_unknown}, we build the proof of Theorem \ref{bound}.  
\begin{proof}[Proof of Theorem \ref{bound} $(a)$]
 Let $T_1=0$, $T_2+\eta=T$, then inserting \eqref{a} to
 \eqref{moment_unknown} straightforwardly yields that
 \begin{equation*}
  \left\|\E[I_t^{1-\alpha} u(x_0,t,\o)]\right\|_{L^1(0,T)}
 \ge \|g_1\|_{L^1(0,T-\eta)}\ \| v(x_0,\cdot)\|_{L^1(0,\eta)}.
 \end{equation*}
For the left side, we have
\begin{equation*}
\begin{aligned}
 \left\|\E[I_t^{1-\alpha} u(x_0,t,\o)]\right\|_{L^1(0,T)}
 &\le \frac{1}{\Gamma(1-\alpha)}\E\Big[\int_0^T 
 \int_0^t (t-\tau)^{-\alpha}|u(x_0,\tau,\o)|\ d\tau\ dt\Big]\\
 &=\frac{1}{\Gamma(2-\alpha)}\E\Big[\int_0^T (T-\tau)^{1-\alpha}
 |u(x_0,\tau,\o)|\ d\tau\Big]\\
 &\le C_\alpha T^{1-\alpha}\E\big[\|u(x_0,\cdot,\o)\|_{L^1(0,T)}\big],
\end{aligned}
\end{equation*}
then
\begin{equation*}
 \|g_1\|_{L^1(0,T-\eta)}\le C_\alpha B_\eta T^{1-\alpha}\ \E\big[\|u(x_0,\cdot,\o)\|_{L^1(0,T)}\big].
\end{equation*}
\end{proof}

\begin{proof}[Proof of Theorem \ref{bound} $(b)$]
Let $\eta>0$ be small and assume that
$g_1$ changes sign on $0<t_1<t_2<\cdots<t_N<T-\eta$, for convenience, we set
$t_0=0$ and $t_{N+1}=T-\eta$. By \eqref{moment_unknown}, for
$t\ge t_k$, we can write
\begin{equation*}
 \int_{t_k}^t g_1(\tau)  v(x_0,t-\tau)\ d\tau
 =\E[I_t^{1-\alpha} u(x_0,t,\o)]-\sum_{j=1}^{k} S_j,
\end{equation*}
where
\begin{equation*}
 S_j=\int_{t_{j-1}}^{t_j} g_1(\tau) v(x_0,t-\tau)\ d\tau.
\end{equation*}
Using \eqref{b} to the above equality with $T_1=t_k,\ T_2=t_{k+1}$, we
can obtain that for $k=0,\cdots, N$,
\begin{equation}\label{inequality_3}
\begin{aligned}
 \|g_1\|_{L^1(t_k,t_{k+1})}\le &B_\eta \Big\|\int_{t_k}^t g_1(\tau)
  v(x_0,t-\tau)\ d\tau\Big\|_{L^1(t_k,t_{k+1}+\eta)}
 +2\|g_1\|_{L^1(t_{k+1},t_{k+1}+\eta)}\\
 \le &B_\eta\Big( \big\|\E[I_t^{1-\alpha} u(x_0,t,\o)]\big\|_{L^1(t_k,t_{k+1}+\eta)}
 +\sum_{j=1}^{k} \|S_j\|_{L^1(t_k,t_{k+1}+\eta)}\Big)\\
 &+2\|g_1\|_{L^1(t_{k+1},t_{k+1}+\eta)}.
\end{aligned}
\end{equation}
For $\|g_1\|_{L^1(t_{k+1},t_{k+1}+\eta)}$, from the condition that
$\|g_1\|_{L^\infty(0,T)}\le M$ we have
\begin{equation}\label{inequality_1}
\begin{aligned}
 \|g_1\|_{L^1(t_{k+1},t_{k+1}+\eta)}
 =&\int_{t_{k+1}}^{t_{k+1}+\eta} |g_1(\tau)|\ d\tau \le M\eta.
 \end{aligned}
\end{equation}
For $\|S_j\|_{L^1(t_k,t_{k+1}+\eta)}$, it holds that
\begin{equation*}
 \begin{aligned}
  \|S_j\|_{L^1(t_k,t_{k+1}+\eta)}
  \le&\int_{t_k}^{t_{k+1}+\eta} \int_{t_{j-1}}^{t_j} |g_1(\tau)|
   v(x_0,t-\tau)\ d\tau\ dt\\
  =&\int_{t_{j-1}}^{t_j} |g_1(\tau)| \int_{t_k}^{t_{k+1}+\eta}
   v(x_0,t-\tau)\ dt\ d\tau\\
  =&\int_{t_{j-1}}^{t_j} |g_1(\tau)|
  \ \| v(x_0,\cdot)\|_{L^1(t_k-\tau,t_{k+1}+\eta-\tau)}\ d\tau.
 \end{aligned}
\end{equation*}
Assumption \ref{assumption} and Lemma \ref{max} give that $|v(x_0,t)|\le C_f$. Consequently,
\begin{equation}\label{inequality_2}
\begin{aligned}
 \|S_j\|_{L^1(t_k,t_{k+1}+\eta)}
 \le & C_fT \|g_1\|_{L^1(t_{j-1},t_j)},\quad j=1,\cdots,k.
\end{aligned}
\end{equation}
Inserting \eqref{inequality_1} and \eqref{inequality_2} into
\eqref{inequality_3} yields that
\begin{equation}\label{inequality_5}
\begin{aligned}
 \|g_1\|_{L^1(t_k,t_{k+1})}\le &
 B_\eta\Big\|\E[I_t^{1-\alpha} u(x_0,t,\o)]\Big\|_{L^1(t_k,t_{k+1}+\eta)} +B_\eta C_fT \|g_1\|_{L^1(0,t_k)}+2M\eta.
\end{aligned}
\end{equation}
Fix $k=0$, we have
\begin{equation}\label{inequality_4}
 \|g_1\|_{L^1(0,t_1)}\le B_\eta\Big\|
 \E[I_t^{1-\alpha} u(x_0,t,\o)]\Big\|_{L^1(0,t_1+\eta)}+2M\eta.
\end{equation}

Now we claim that for $k=1,\cdots,N+1$,
\begin{equation*}
 \|g_1\|_{L^1(0,t_k)}\le \frac{(B_\eta C_fT+1)^k-1}{B_\eta C_fT}
 \Big(B_\eta\big\|\E[I_t^{1-\alpha} u(x_0,t,\o)]
 \big\|_{L^1(0,t_k+\eta)}+2M\eta\Big),
\end{equation*}
and prove it by induction. The case of $k=1$ is valid by \eqref{inequality_4}.
Now assume that the claim holds for $k=l$, then for $k=l+1$, the estimate \eqref{inequality_5} gives that
\begin{equation*}
 \begin{aligned}
  \|g_1\|_{L^1(0,t_{l+1})}\le& \|g_1\|_{L^1(0,t_l)}
  + B_\eta\big\|\E[I_t^{1-\alpha} u(x_0,t,\o)]\big\|_{L^1(t_l,t_{l+1}+\eta)}
 +B_\eta C_fT \|g_1\|_{L^1(0,t_l)}+2M\eta\\
 \le&  (B_\eta C_fT+1) \frac{(B_\eta C_fT+1)^l-1}{B_\eta C_fT}
 \Big(B_\eta\big\|\E[I_t^{1-\alpha} u(x_0,t,\o)]
 \big\|_{L^1(0,t_l+\eta)}+2M\eta\Big)\\
 &+B_\eta\big\|\E[I_t^{1-\alpha} u(x_0,t,\o)]\big\|_{L^1(t_l,t_{l+1}+\eta)} +2M\eta\\
 \le & \frac{(B_\eta C_fT+1)^{l+1}-1}{B_\eta C_fT}
 \Big(B_\eta\big\|\E[I_t^{1-\alpha} u(x_0,t,\o)]
 \big\|_{L^1(0,t_{l+1}+\eta)}+2M\eta\Big).
 \end{aligned}
\end{equation*}
So the claim is valid, and recalling that $t_{N+1}=T-\eta$, we have
\begin{equation*}
\begin{aligned}
 \|g_1\|_{L^1(0,T-\eta)}\le &\frac{(B_\eta C_fT+1)^{N+1}-1}{B_\eta C_fT} \Big(B_\eta\big\|\E[I_t^{1-\alpha} u(x_0,t,\o)]
 \big\|_{L^1(0,T)}+2M\eta\Big)\\
 \le&\frac{(B_\eta C_fT+1)^{N+1}-1}{B_\eta C_fT}
 \left(C_\alpha B_\eta T^{1-\alpha}\ \E\big[\|u(x_0,\cdot,\o)\|_{L^1(0,T)}\big]+2M\eta\right).
 \end{aligned}
\end{equation*}
The proof is complete.
\end{proof}

\begin{proof}[Proof of Theorem \ref{bound} $(c)$]
Note that $g_2^2$ keeps its sign on $(0,T)$.
Analogous to the proof of Theorem \ref{bound} $(a)$, setting $T_1=0$, $T_2=T-\eta$ in \eqref{a}, then \eqref{moment_unknown} gives that
\begin{equation*}
\begin{aligned}
 \big\|\V [I_t^{1-\alpha} u(x_0,t,\o)]\big\|_{L^1(0,T)}
 &\ge\|g_2^2\|_{L^1(0,T-\eta)}\ \|[ v(x_0,\cdot)]^2\|_{L^1(0,\eta)}\\
 &=\|g_2\|^2_{L^2(0,T-\eta)}\ \| v(x_0,\cdot)\|^2_{L^2(0,\eta)}.
 \end{aligned}
\end{equation*}
Holder inequality yields that
\begin{equation*}
 \| v(x_0,\cdot)\|_{L^2(0,\eta)}\ge \|1\|^{-1}_{L^2(0,\eta)}
\| v(x_0,\cdot)\|_{L^1(0,\eta)}=\eta^{-1/2}B_\eta^{-1}.
\end{equation*}
Consequently,
\begin{equation*}
 \|g_2\|_{L^2(0,T-\eta)}\le \eta^{1/2} B_\eta
 \ \big\|\V[I_t^{1-\alpha} u(x_0,t,\o)]\big\|^{1/2}_{L^1(0,T)}.
\end{equation*}
The proof of Theorem \ref{bound} is complete. 
\end{proof}

\section{Numerical reconstruction}\label{section:numerical}

\subsection{Regularized Levenberg-Marquardt iteration} 
The discretized formulation of integral equation \eqref{moment_unknown} 
is derived as follows.  
Denote the uniform mesh on the interval $[0,T]$ as 
$\{0=t_0<t_1<\cdots<t_N=T\}$ and set $\Delta_t=T/N$. 
From Lemmas \ref{max} and \ref{strong_max}, we have $\{t\in[0,T]:v(x_0,t)=0\}$ 
is at most a finite set. Thus we can set 
\begin{equation}\label{positive_v}
 v(x_0,t_1)>0,
\end{equation}
if the mesh size $\Delta_t$ is chosen appropriately. 

Define 
\begin{equation*}
 E(t_n)=\E [I_t^{1-\alpha}u(x_0,t_n,\o)],\quad V(t_n)=\V[I_t^{1-\alpha}u(x_0,t_n,\o)].
\end{equation*}
Then from \eqref{moment_unknown} we have 
\begin{equation*}
 \begin{aligned}
  E(t_n) &= \int_0^{t_n} g_1(\tau) v(x_0,t_n - \tau)\ d\tau
    = \sum_{k=1}^n \int_{t_{k-1}}^{t_k} g_1(\tau) v(x_0,t_n - \tau)\ d\tau\\
    &\approx \Delta_t \sum_{k=1}^n [g_1(t_{k-1})v(x_0,t_n - t_{k-1}) 
    + g_1(t_k) v(x_0,t_n - t_k)]/2\\
    &=\Delta_t\Big[g_1(0)v(x_0,t_n)/2+g_1(t_n)v(x_0,0)/2
    +\sum_{k=1}^{n-1} g_1(t_k) v(x_0,t_{n-k})\Big]\\
    &=\Delta_t\Big[g_1(0)v(x_0,t_n)/2
    +\sum_{k=1}^{n-1} g_1(t_k) v(x_0,t_{n-k})\Big],
 \end{aligned}
\end{equation*}
where the last equality comes from $v(x_0,0)=f(x_0)=0$. Analogously, 
\begin{equation*}
  V(t_n) =\Delta_t\Big[g_2^2(0)v^2(x_0,t_n)/2
  +\sum_{k=1}^{n-1} g_2^2(t_k) v^2(x_0,t_{n-k})\Big].
\end{equation*}
From the above results, we can give the discretized formulation of 
\eqref{moment_unknown},
\begin{equation}\label{numerical_equation}
 A_1 \vec{g}_1=\vec{E},\quad A_2 \vec{g}_2=\vec{V}, 
\end{equation}
where 
\begin{equation*}
 \vec{g}_1=
 \begin{bmatrix}
  g_1(t_0)\\ \vdots\\g_1(t_{N-1})           
 \end{bmatrix},
 \quad \vec{g}_2=
 \begin{bmatrix}
  g_2^2(t_0)\\ \vdots\\g_2^2(t_{N-1})           
 \end{bmatrix},
 \quad \vec{E}=
 \begin{bmatrix}
  E(t_1)\\ \vdots\\E(t_N)           
 \end{bmatrix},
 \quad \vec{V}=
 \begin{bmatrix}
  V(t_1)\\ \vdots\\V(t_N)           
 \end{bmatrix},
\end{equation*}
and the matrices $A_1, A_2$ are given as 
\begin{equation*}
 A_1=\Delta_t
 \begin{bmatrix}
      v(x_0,t_1)/2\\
      v(x_0,t_2)/2& v(x_0,t_1) \\
      \vdots&\vdots&\ddots\\
      v(x_0,t_N)/2&v(x_0,t_{N-1})&\cdots& v(x_0,t_1)
 \end{bmatrix},
 \end{equation*}
 \begin{equation*}
 A_2=\Delta_t 
 \begin{bmatrix}
      v^2(x_0,t_1)/2\\
      v^2(x_0,t_2)/2& v^2(x_0,t_1) \\
      \vdots&\vdots&\ddots\\
      v^2(x_0,t_N)/2&v^2(x_0,t_{N-1})&\cdots& v^2(x_0,t_1)
 \end{bmatrix}.
\end{equation*}
From the definitions of $\vec{g}_1,\vec{g}_2$, we can only recover  the unknowns on the partial interval $[0,t_{N-1}]=[0,T-\Delta_t]$, which indicates Theorem \ref{bound}. 

In practice, considering the measured error, the noisy moments $\vec E_\delta, \vec V_\delta$ will be used instead of $\vec E, \vec V$, and  it holds that 
$\|(\vec E_\delta-\vec E)/\vec E\|_{\infty}\le \delta,
\ \|(\vec V_\delta-\vec V)/\vec V\|_{\infty}\le \delta.$
Due to the ill-posedness of this inverse problem, we choose the regularized Levenberg-Marquardt iteration \cite{Levenberg:1944,Marquardt:1963,More:1978} to solve the noisy equation \eqref{numerical_equation}, in which $\vec E_\delta, \vec V_\delta$ are used. The iteration is given as follows,   
\begin{equation}\label{iteration}
\begin{aligned}
 \vec g_{1,k+1}&=\vec g_{1,k}+(A_1^T A_1+\gamma_1 I)^{-1} A_1^T(\vec E_\delta-
 A_1 \vec g_{1,k})\\
 &=B_{1,\gamma_1}\vec g_{1,k}+(A_1^T A_1+\gamma_1 I)^{-1} 
 A_1^T\vec E_\delta,\\
 \vec g_{2,k+1}&=\vec g_{2,k}+(A_2^T A_2+\gamma_2 I)^{-1} 
 A_2^T(\vec V_\delta- A_2 \vec g_{2,k})\\
 &=B_{2,\gamma_2}\vec g_{2,k}+(A_2^T A_2+\gamma_2 I)^{-1} 
 A_2^T\vec V_\delta,
 \end{aligned}
\end{equation}
where  
$$
B_{l,\gamma_l}=I-(A_l^T A_l+\gamma_l I)^{-1} A_l^T A_l,\ \ l=1,2,
$$
and the regularization parameters $\gamma_1, \gamma_2$ are chosen as 
small positive constants.

\subsection{Convergence of iteration \eqref{iteration}}
The spectral radius of a square matrix, which is denoted by $\rho(\cdot)$,  
is defined as the largest absolute value of its eigenvalues. 
The next lemma concerns the spectral radius of the matrices 
$B_{l,\gamma_l}, \ l=1,2$, and after that the convergence of \eqref{iteration} is proved. 
\begin{lemma}\label{spectral radius}
 $\rho(B_{l,\gamma_l})<1,\ l=1,2.$ 
\end{lemma}
\begin{proof}
 Let $(\lambda,\vec y)$ be one eigenpair of $B_{1,\gamma_1}$ with 
 $\vec y\ne \vec 0$. We need to show that $|\lambda|<1$. 
 
 From $B_{1,\gamma_1}\vec y=\lambda \vec y$, we can deduce that 
 \begin{equation*}
  \lambda(A_1^TA_1+\gamma_1 I)\vec y=\gamma_1 \vec y.
 \end{equation*}
Taking inner product with $\vec y$ yields that 
$$\gamma_1 \l\vec y,\vec y\rangle=\lambda\l(A_1^TA_1+\gamma_1 I)\vec y,\vec y\rangle
=\lambda\big(\l A_1\vec y,A_1\vec y\rangle+\gamma_1 \l\vec y,\vec y\rangle\big),$$
which gives 
$$|\lambda|=\Big|\frac{\gamma_1 \| \vec y\|_2^2}{\| A_1\vec y\|_2^2
+\gamma_1 \|\vec y\|_2^2}\Big|.$$ 
Condition \eqref{positive_v} ensures 
the invertibility of $A_1$, which together with $\vec y\ne \vec 0$ gives $\| A_1\vec y\|_2^2>0$. Hence, considering that $\gamma_1$ is chosen as a small positive constant, we have 
$$0<\gamma_1 \| \vec y\|_2^2<\| A_1\vec y\|_2^2
+\gamma_1 \|\vec y\|_2^2,$$ 
which yields $|\lambda|<1$. 

The case for $B_{2,\gamma_2}$ can be proved analogously. The proof is complete. 
\end{proof}

\begin{proposition}\label{convergence}
 The sequences $\{ \vec g_{1,k}\}_{k=0}^\infty, \{\vec g_{2,k}\}_{k=0}^\infty$ generated from 
 iteration \eqref{iteration} are both convergent. Also, if we denote the 
 limits by $\vec g_{1,\delta}$ and $ \vec g_{2,\delta}$, respectively, 
 then 
 $$
 \lim_{\delta\to0+}\|\vec g_{1,\delta}- \vec g_1\|_{\infty}
 =\lim_{\delta\to0+}\|\vec g_{2,\delta}- \vec g_2\|_{\infty}=0, 
 $$
 where $ \vec g_1,\vec g_2$ solve equation \eqref{numerical_equation}.
\end{proposition}
\begin{proof}
 The convergence follows from Lemma \ref{spectral radius} straightforwardly. 
 
 From \eqref{iteration}, we have 
 \begin{equation*}
  \vec g_{1,\delta}=(A_1^TA_1)^{-1}A_1^T \vec E_{\delta},
  \ \vec g_{2,\delta}=(A_2^TA_2)^{-1}A_2^T \vec V_{\delta}.
 \end{equation*}
Considering the results $ \vec g_1= (A_1^TA_1)^{-1}A_1^T \vec E,
\ \vec g_2=(A_2^TA_2)^{-1}A_2^T\vec V$ and 
$$\lim_{\delta\to0+}\| \vec E-\vec E_{\delta}\|_{\infty}
=\lim_{\delta\to0+} \| \vec V-\vec V_{\delta}\|_{\infty}=0,$$ 
it follows that 
$\vec g_{l,\delta}\to \vec g_l,\ l=1,2$ in the sense of $\|\cdot\|_\infty$ as $\delta\to 0+$.
\end{proof}

\subsection{Forward problem solver}
To obtain the measurements, 
the direct problem of equation \eqref{SDE} should be considered. 
We first introduce the finite element method, and then
the GMsFEM\cite{gmsfem} to handle the case that $\kappa$ is highly heterogeneous.

On space $x$, the piecewise linear basis $\{\phi_j(x) \}_{j=1}^m$ is used with $\phi_j=0$ on $\partial D$ and $ \phi_j(x_k) = \delta_{jk},$ 
where $\{ x_j \}_{j=1}^m$ consist of a Delaunay triangulation $\mathcal{T}^h$ on domain $D$, and $h>0$ is the fine mesh size. Then we define the finite element space as $\Vc_m=\text{span}\{\phi_j(x):j=1,\cdots,m\}$. With this space, the 
projection operator $P_m:L^2(D)\mapsto \Vc_m$ can be given as 
\begin{equation*}
 P_m \psi(x)= \sum_{j=1}^{m} \psi(x_j) \phi_j(x) =:\tilde{\psi}(x). 
\end{equation*}
Here we use the notation $\tilde{\psi}$ for short and the corresponding 
vector form is denoted by $\vec{\psi}=[\psi(x_j)]_{j=1}^m$. 

Writing 
\begin{align*}
   \tilde{u}(x,t,\o) = \sum_{j=1}^m u(x_j,t,\o) \phi_j(x),
   \quad \tilde{f}(x) = \sum_{j=1}^m f(x_j) \phi_j(x),
\end{align*}
then the weak formulation of equation \eqref{SDE} with test function 
$\phi_i$ is 
\begin{equation*}\label{eq:fem_SDE}
\sum_{j=1}^m \partial_t^\alpha u(x_j,t,\o) \l\phi_j, \phi_i\rd
+  \sum_{j=1}^m u(x_j,t,\o) \l\A \phi_j, \phi_i\rd = 
g(t,\o) \sum_{j=1}^m f(x_j) \l\phi_j, \phi_i\rd.
\end{equation*}
From the above formulation, we define the mass matrix $\vec{M}$ 
and stiff matrix $\vec{S}$ w.r.t. basis $\{\phi_j\}_{j=1}^m $ as
\begin{align*}
    \vec{M} = \Big[\l\phi_j, \phi_i\rd \Big]_{i,j = 1}^m , 
    \quad   \vec{S} = \Big[ \l\A \phi_j, \phi_i\rd\Big]_{i,j = 1}^m ,
\end{align*}
which will be used to construct the discretized scheme for equation 
\eqref{SDE}.

For the discretization on time, the $L_1$-stepping scheme \cite{JinLazarovZhou:2016,RundellZhang:2018} is used. Again we use the uniform time mesh $0 = t_0 < t_1 <\cdots < t_N = T$ and 
denote $\Delta_t = T/N$. Then the fractional derivative $\partial_t^\alpha$ is approximated as
\begin{align*}
    & \partial_t^\alpha \psi(t_1) \approx b_{1,0} (\psi(t_1) - \psi(t_0)),\\
    & \partial_t^\alpha \psi(t_n) \approx \sum_{k=1}^{n-1} (b_{n,k-1} - b_{n,k}) \psi(t_k) 
    + b_{n,n-1}\psi(t_n) - b_{n,0} \psi(t_0) , 
    \quad n=2,\cdots,N,
\end{align*}
with parameters
$$ b_{n,k} = \Gamma(2-\alpha)^{-1} \Delta_t^{-\alpha} [(n-k)^{1-\alpha} - (n-k-1)^{1-\alpha}] ,\quad k = 0,\cdots,n-1.$$

For the random term $g(t_n,\o) = g_1(t_n) + g_2(t_n) \dot \W(t_n)$, due to $\W(t)-\W(s)\sim \mathcal{N}(0,t-s),$ we have 
$$ \dot \W(t_n) \approx [\W(t_n)-\W(t_{n-1})]/\Delta_t \sim 
\Delta_t^{-1/2}\mathcal{N}(0,1).$$

Hence, considering the vanishing initial condition, the finite element scheme for solving equation \eqref{SDE} is given as: for $n=1,\cdots,N$, find the vector form $\vec{u}_n$ of $\tilde u(x,t_n,\o)\in \Vc_m$ such that  
\begin{equation}\label{eq:SDE_discrete}
\begin{aligned}
\left( b_{1,0}\vec{M} + \vec{S} \right) \vec{u}_1
=& \vec{M}  \Big( g_1(t_1)\vec f+ g_2(t_1) 
\Delta_t^{-1/2}\mathcal{N}(0,1)\vec f\Big),\\
\left( b_{n,n-1}\vec{M} + \vec{S} \right) \vec{u}_n
=& \vec{M}  \Big( g_1(t_n)\vec f+ g_2(t_n)\Delta_t^{-1/2} \mathcal{N}(0,1)\vec f
+\sum_{k=1}^{n-1}(b_{n,k}-b_{n,k-1})\vec{u}_k \Big).
\end{aligned}
\end{equation}
Following this scheme, the solution $v(x,t)$ of equation \eqref{v} 
can be simulated similarly.

\subsubsection{GMsFEM}\label{gmsfem}
In many practical applications, the coefficient $\kappa(x)$ 
can be highly heterogeneous, in which very fine mesh is required in the finite element method, accompanied with huge computational cost. So we choose the Generalized Multiscale Finite Element Method (GMsFEM \cite{gmsfem}) as the model reduction technique here. 
The GMsFEM provides a systematic way of reducing the 
computational cost in solving various types of highly heterogeneous partial differential equations \cite{fu2017fast,galvis2015generalized,vasilyeva2020multiscale,chung2015mixed}. 
This method reduces the degrees of freedom of large systems by constructing
appropriate multiscale basis functions, which are only needed to calculate one time. Therefore GMsFEM is particular suitable for the computing that requires solving a fixed equation repetitively but with different source or boundary conditions. Besides, by choosing different number of basis, we can easily tune the accuracy of the solution, which may be useful in inverse problems based on the recent
research \cite{lessismore}.  

In GMsFEM, there are two stages to construct the generalized multiscale basis: the snapshot stage and offline stage.
\begin{figure}[h]
	\centering
		\includegraphics[width=3in,height=2in]{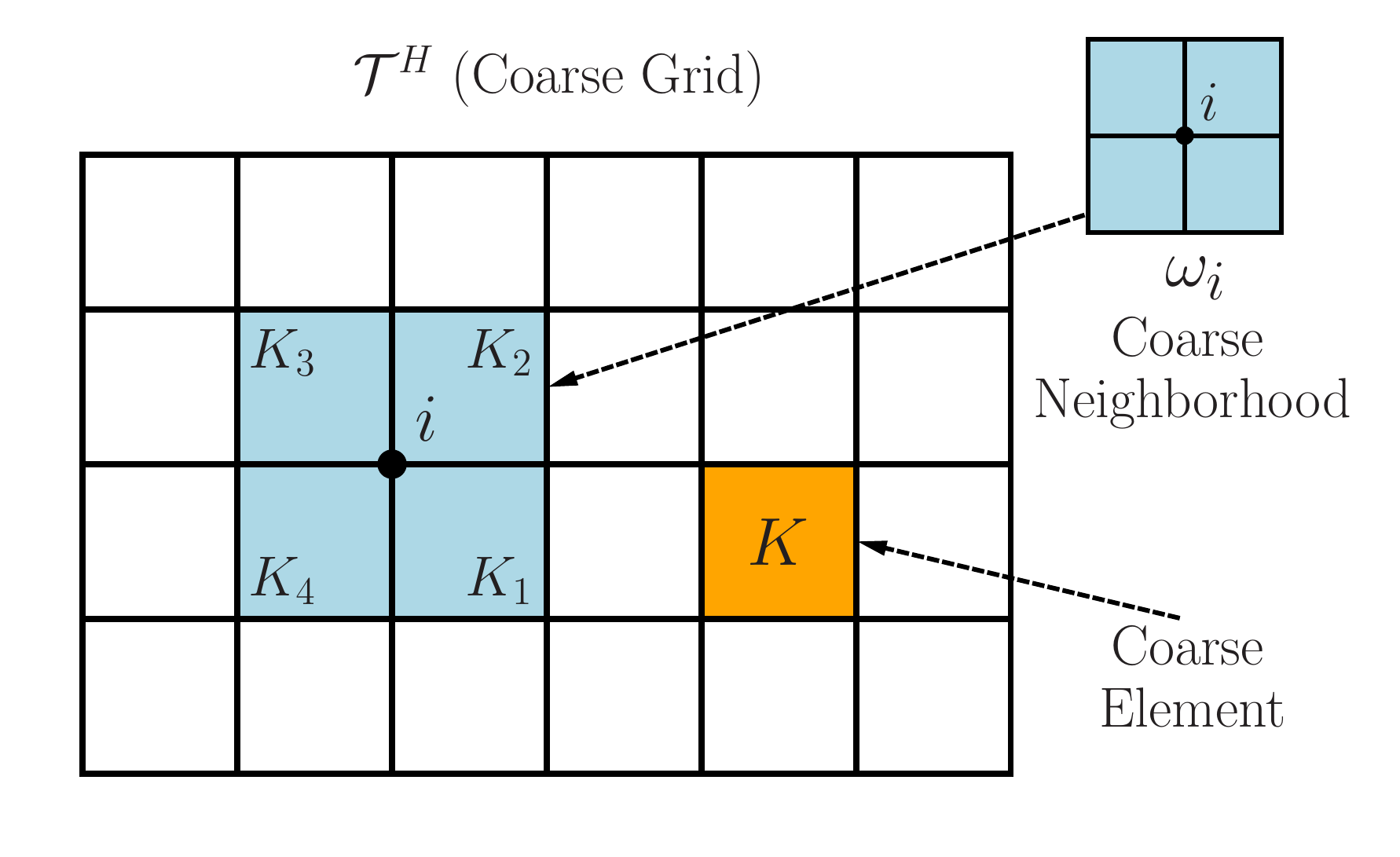}
		\caption{Illustration of coarse neighborhood and coarse element.}
		\label{fig:grid}
\end{figure}
We consider a triangulation of domain $D$ denoted by  $\mathcal{T}^H$
such that  $\mathcal{T}^h$ is its refinement. 
Let $\mathcal{S}^H$ be the set of all coarse grid nodes and $N_S=|\mathcal{S}^H|$. Elements of $\mathcal{T}^H$ are called coarse grid blocks.
For each vertex $ {x}_i \in \mathcal{S}^H$ in the grid  $\mathcal{T}^H$, we define the coarse neighborhood $\omega_i$ by
\begin{equation*}
\omega_i = \bigcup \{ K_j \; : \; K_j \subset \mathcal{T}^H, \;  {x}_i \in K_j \}.
\end{equation*}
That is, $\omega_i$ is the union of all coarse grid blocks $K_j$
containing the vertex $ {x}_i$, see Figure \ref{fig:grid}. We will construct multiscale basis functions
in each coarse neighborhood $\omega_i$. 

We begin by the construction of local snapshot spaces in $\omega_i$.
There are two types of local snapshot spaces.
The first type is
\begin{equation*}
\Vc_1^{i,\text{snap}} = \Vc_m(\omega_i),
\end{equation*}
where $\Vc_m(\omega_i)$ is the restriction of the $\Vc_m$ to $\omega_i$. 
Therefore, $\Vc_1^{i,\text{snap}}$ contains all possible fine scale functions defined on $\omega_i$. The second type is the harmonic extension space. More specifically, let $\Vc_m(\partial\omega_i)$ be the restriction of the conforming space to $\partial\omega_i$.
Then we define the fine-grid delta function $\delta_k \in \Vc_m(\partial\omega_i)$ on $\partial\omega_i$ by
\begin{equation*}
\delta_k( {x}_l) = 
\begin{cases}
1, \quad & l = k, \\
0, \quad & l \ne k,
\end{cases}
\end{equation*}
where $ \{{x}_l\}$ are all fine grid nodes on $\partial\omega_i$. 
Given $\delta_k$, we seek $ {u}_{k}$ by
\begin{equation}
\begin{aligned}
- \A u_{k} &=  {0}, &&\text{in} \ \omega_i, \\
{u}_{k} &= \delta_k, &&\text{on} \ \partial\omega_i.
\end{aligned}
\label{eq:cg_snap_har}
\end{equation}
The linear span of the above harmonic extensions is our second type local snapshot space $\Vc^{i,\text{snap}}_2$. 
To simplify the presentations, we will use $\Vc^{i,\text{snap}}$ to denote $\Vc^{i,\text{snap}}_1$ or $\Vc^{i,\text{snap}}_2$
when there is no need to distinguish them. Moreover, we write
\begin{equation*}
\Vc^{i,\text{snap}} = \text{span} \{  {\psi}^{i,\text{snap}}_k: k=1,2,\cdots, M^{i,\text{snap}} \},
\end{equation*}
where ${\psi}^{i,\text{snap}}_k$ is the snapshot functions, and 
$M^{i,\text{snap}}$ is the number of basis functions in $\Vc^{i,\text{snap}}$. 

The dimension of the snapshot space is still too large for computation.
We can use a spectral problem to select the dominant modes from
the snapshot space. 
Specifically, in each neighborhood, we consider
\begin{equation}\label{eq:spec-cg}
\A \phi=\lambda \tilde{\kappa} \phi, 
\end{equation}
where $\tilde{\kappa}={\kappa}\sum_{i=1}^{N_S} | \nabla \chi_i |^2,$  
$N_S$ is the total number of neighborhoods, and $\chi_i$ is
the partition of unity function \cite{pu} for $\omega_i$.
One choice of a partition of unity function is the coarse grid hat function whose value at the coarse vertex $x_i$ is 1 and 0 at all other coarse vertices. Another choice of the basis function is introduced in \cite{hou1997}.
We solve the above spectral problem (\ref{eq:spec-cg}) in the local snapshot space 
$\Vc^{i,\text{snap}}$. 
Then we use the first $L_i$ eigenfunctions $\phi_i$ corresponding to the first $L_i$ eigenvalues  to construct the local offline space. 
We define
\begin{equation*}
{\psi}^{i,\text{off}}_l = \sum_{k=1}^{M^{i,\text{snap}}} \phi_{l,k}  {\psi}^{i,\text{snap}}_k, \quad\quad l=1,2,\cdots, L_i,
\end{equation*}
where $\phi_{l,k}$ is the $k$-th component of $\phi_l$. 
Note that the function ${\psi}^{i,\text{off}}_l$ is not globally
continuous, therefore we need to multiply it with the partition of unity function. We define the local offline space as
\begin{equation*}
\Vc^{i,\text{off}}_H = \text{span} \{\chi_i{\psi}^{i,\text{off}}_l: l=1,2,\cdots, L_i \}.
\end{equation*}
Then, the offline space can be defined as
\begin{equation*}
\Vc^{\text{off}}_H = \text{span} \{\Vc^{i,\text{off}}_H: i=1,2,\cdots, N_S \}.
\end{equation*}
We note that equations \eqref{eq:cg_snap_har}, \eqref{eq:spec-cg} are solved on the fine grid $\mathcal{T}^h$ numerically.
Then we can treat each discrete offline basis in $\Vc^{\text{off}}$
as a column vector $\vec\Phi_i$, and denote $\vec R = [\vec \Phi_1,\cdots, \vec \Phi_L ]$ be the matrix that represents all the multiscale
basis functions (total number $L=\sum_{i}^{N_S}L_i$).
Thus, the discretized scheme for solving equation \eqref{SDE} with GMsFEM is given as: for $\tilde{u}_{H,n}\in \Vc^{\text{off}}_H,\ n=1,\cdots,N,$ its 
vector form $\vec{u}_{H,n}$ satisfies 
\begin{equation}\label{eq:SDE_discrete_ms}
\begin{aligned}
\left( b_{1,0}\vec{M}_H + \vec{S}_H \right) \vec{u}_{H,1}
=& \vec{M}_H  \Big( g_1(t_1)\vec f_H+ g_2(t_1) 
\Delta_t^{-1/2}\mathcal{N}(0,1)\vec f_H\Big),\\
\left( b_{n,n-1}\vec{M}_H + \vec{S}_H \right) \vec{u}_{H,n}
=& \vec{M}_H  \Big( g_1(t_n)\vec f_H+ g_2(t_n)\Delta_t^{-1/2} \mathcal{N}(0,1)\vec f_H\\
&\quad\quad+\sum_{k=1}^{n-1}(b_{n,k}-b_{n,k-1})\vec{u}_{H,k} \Big),
\end{aligned}
\end{equation}
where $\vec{M}_H={\vec{R}}^T\vec{M}\vec{R},\ \vec{S}_H=\vec{R}^T\vec{S}\vec{R},\ \vec{f}_H=\vec{R}^T\vec f$.
Typically, we only need to select a few numbers of basis in a neighborhood, which ensures that the degrees of freedom of scheme  \eqref{eq:SDE_discrete_ms} is much smaller comparing with scheme \eqref{eq:SDE_discrete}.
After we obtain $\vec{u}_{H,n}$, one projects the solution into the space 
$\Vc_m$ through $\vec{u}_{H,n}^h=\vec R\vec{u}_{H,n}$.

\subsection{Numerical experiments}

Now we present several numerical experiments to show the performance of our algorithm. 
We first consider the smooth case, 
$$g_1(t)=t+\sin(2\pi t)+\sin(3\pi t),\quad g_2(t)=0.5t+\sin(\pi t)-\sin(2\pi t).$$
We set $T=1$, the spatial component $f(x)$ in source term is shown in Figure \ref{fig:source}, and the observation point is chosen as $x_0=(0.4,0.2)$, which is out of $\text{supp}(f)$. In addition, $3\times 10^4$ realizations 
of the single point data $u(x_0,t,\o)$ are recorded, and $1\%$ relative noise is added on the moments, i.e. $\delta=1\%$.  
\begin{figure}[H]
	\centering
	\subfigure[$f(x)$]{
		\includegraphics[width=3in,height=2.1in]{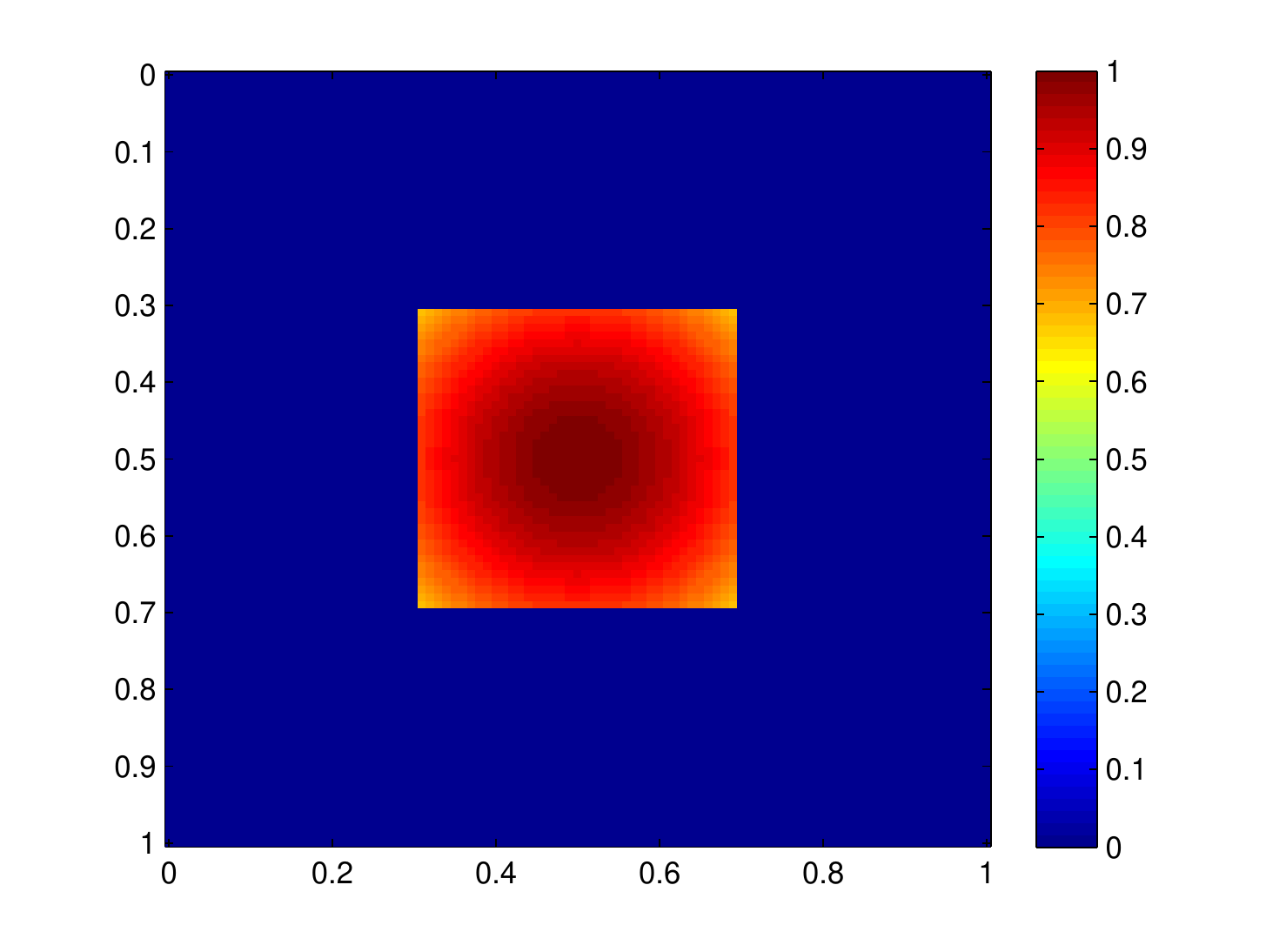}}	
	\caption{Spatial component $f(x).$}
	\label{fig:source}
\end{figure}

We consider both the homogeneous and highly heterogeneous media cases.
For the homogeneous case (test model 1), the model size is $50\times 50$. We apply FEM \eqref{eq:SDE_discrete} for 
the forward modeling. The corresponding inversion results
 are presented in Figure \ref{fig:homo}. We can see our inversion algorithm \eqref{iteration} can generate satisfactory approximations of the targeted unknowns. 

Two heterogeneous experiments are considered and the corresponding $\kappa(x)$ are shown in Figure \ref{fig:model}. The size of these models is $100\times 100$, and in GMsFEM \eqref{eq:SDE_discrete_ms}, we use a $10\times 10$ coarse grid. 
Therefore, the degrees of freedom for the FEM system is 9801 and it is 
242 for the GMsFEM with 2 bases, we can see huge reduction of the unknowns in forward modeling.
The inversion results for the heterogeneous case are displayed in Figures \ref{fig:het} and \ref{fig:cross}.
The comparisons of the approximations for $g_1$ and $|g_2|$ from 
the fine-grid FEM, the GMsFEM with two bases and the 
GMsFEM with one basis are displayed. It can be seen clearly that the GMsFEM with only one basis yields unjustifiable inversion results especially for test model 2. However, the results from GMsFEM with two bases is can be comparable with the FEM results and it is better than FEM for the approximation of $|g_2|$. This is not surprising according to \cite{lessismore}, which tells us that more accurate forward modeling will not always yield better inversion performance. Also we note that the running time of FEM is about 15 times of GMsFEM, and more computational time saving is expected if the size of the model is larger. In addition, GMsFEM with one basis is actually the MsFEM basis \cite{hou1997}, which is more suitable for highly oscillating media, it is an ideal choice to use spectral basis space for high-contrast media inversion.

\begin{figure}[H]
	\centering
	\subfigure[Approximation comparison for $g_1$]{
		\includegraphics[width=2.8in,height=2.1in]{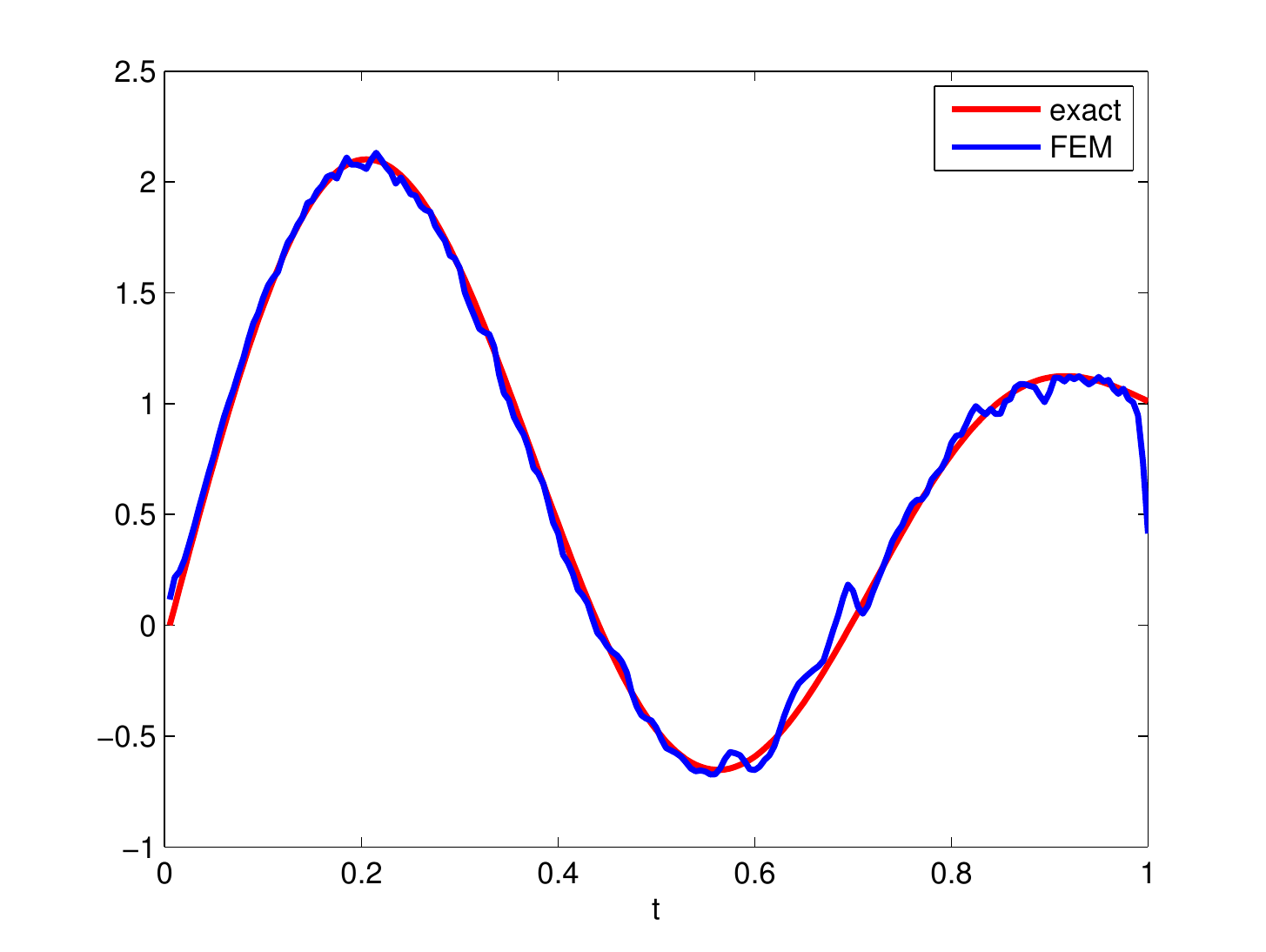}}	
	\subfigure[Approximation comparison for $|g_2|$]{
		\includegraphics[width=2.8in,height=2.1in]{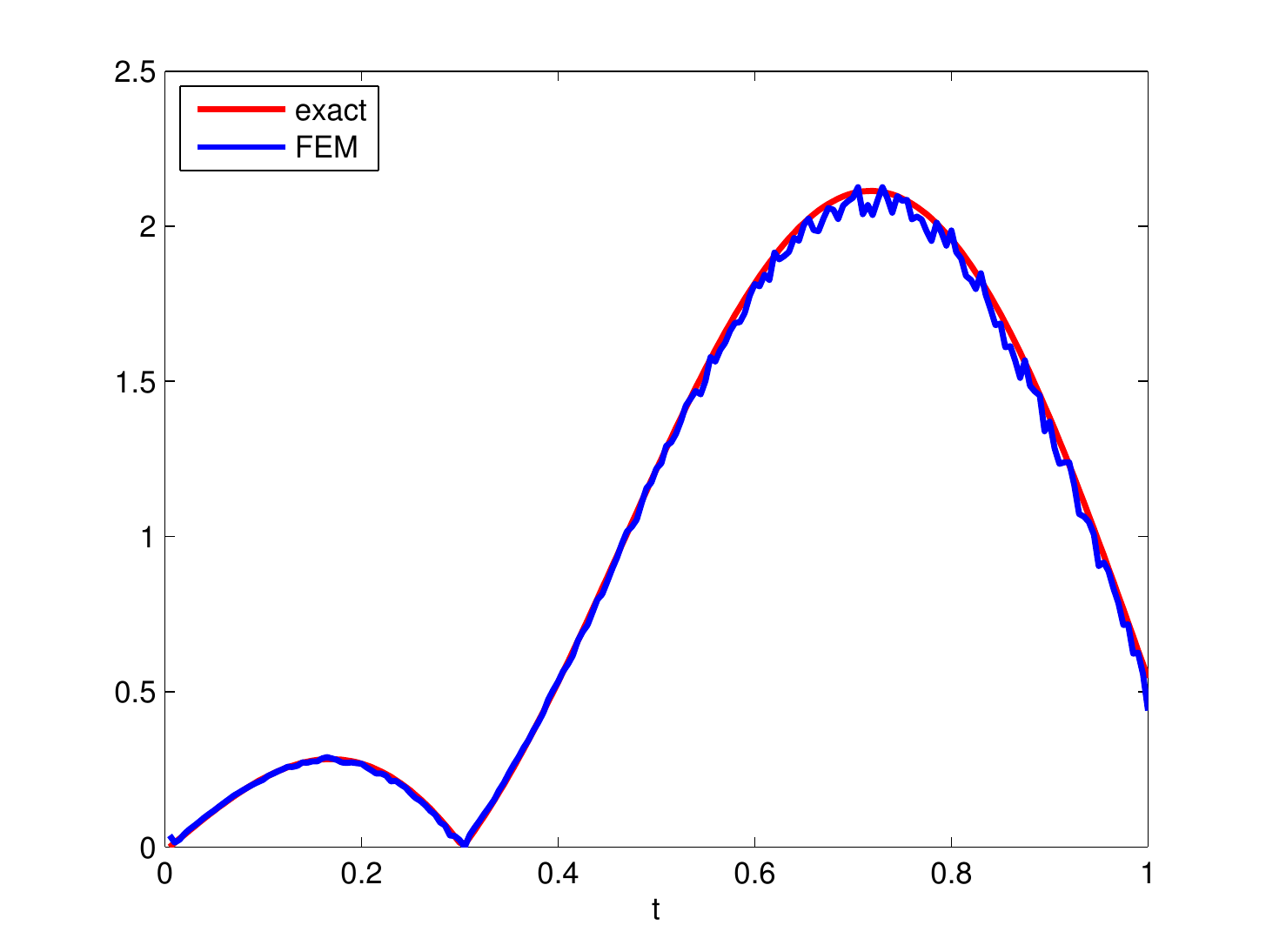}}
	\caption{Results for homogeneous model (test model 1), smooth case.}
	\label{fig:homo}
\end{figure}

\begin{figure}[H]
	\centering
	\subfigure[test model 2]{
		\includegraphics[width=2.9in,height=2.1in]{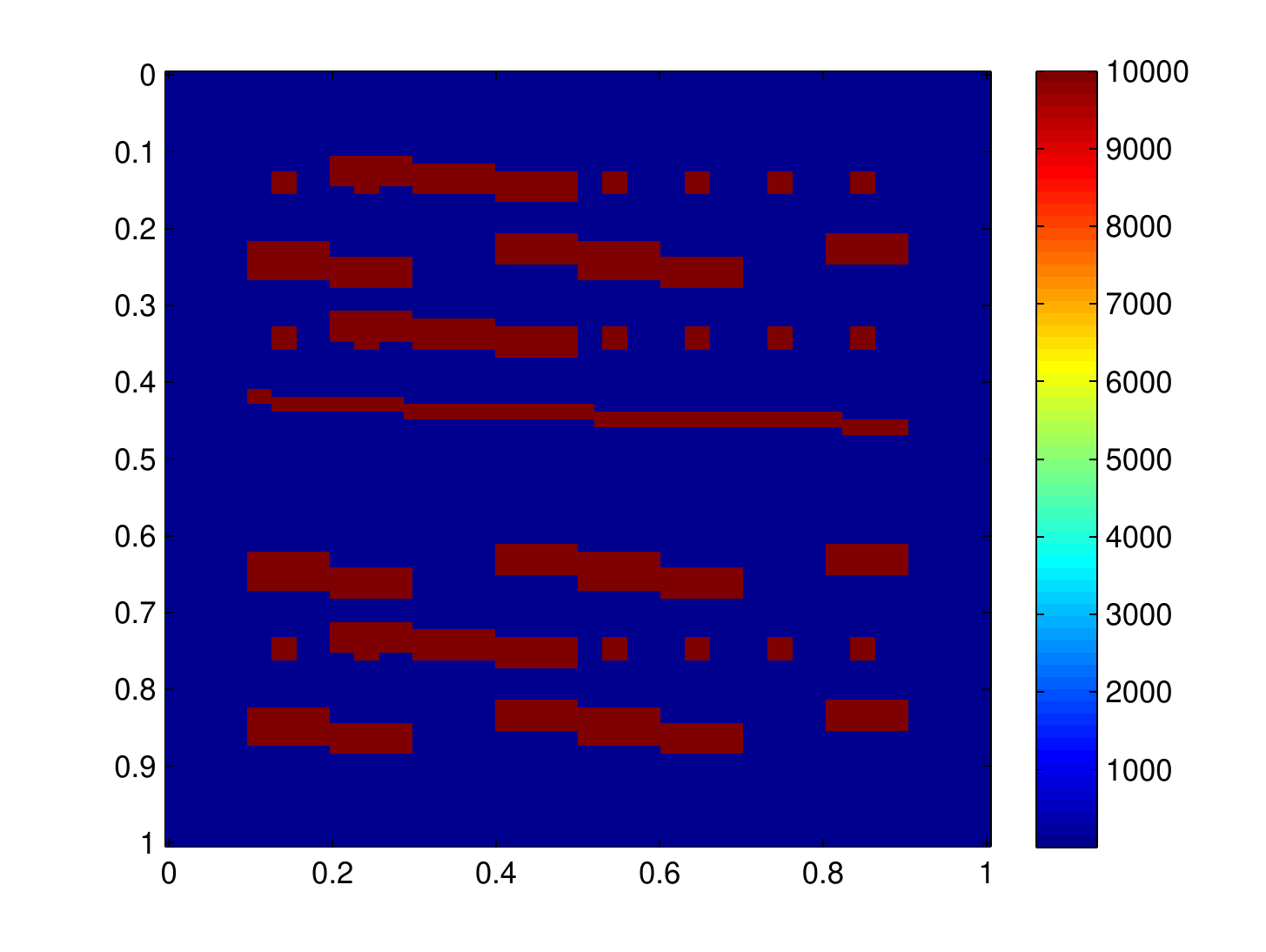}}
	\subfigure[test model 3]{
	\includegraphics[width=2.9in,height=2.1in]{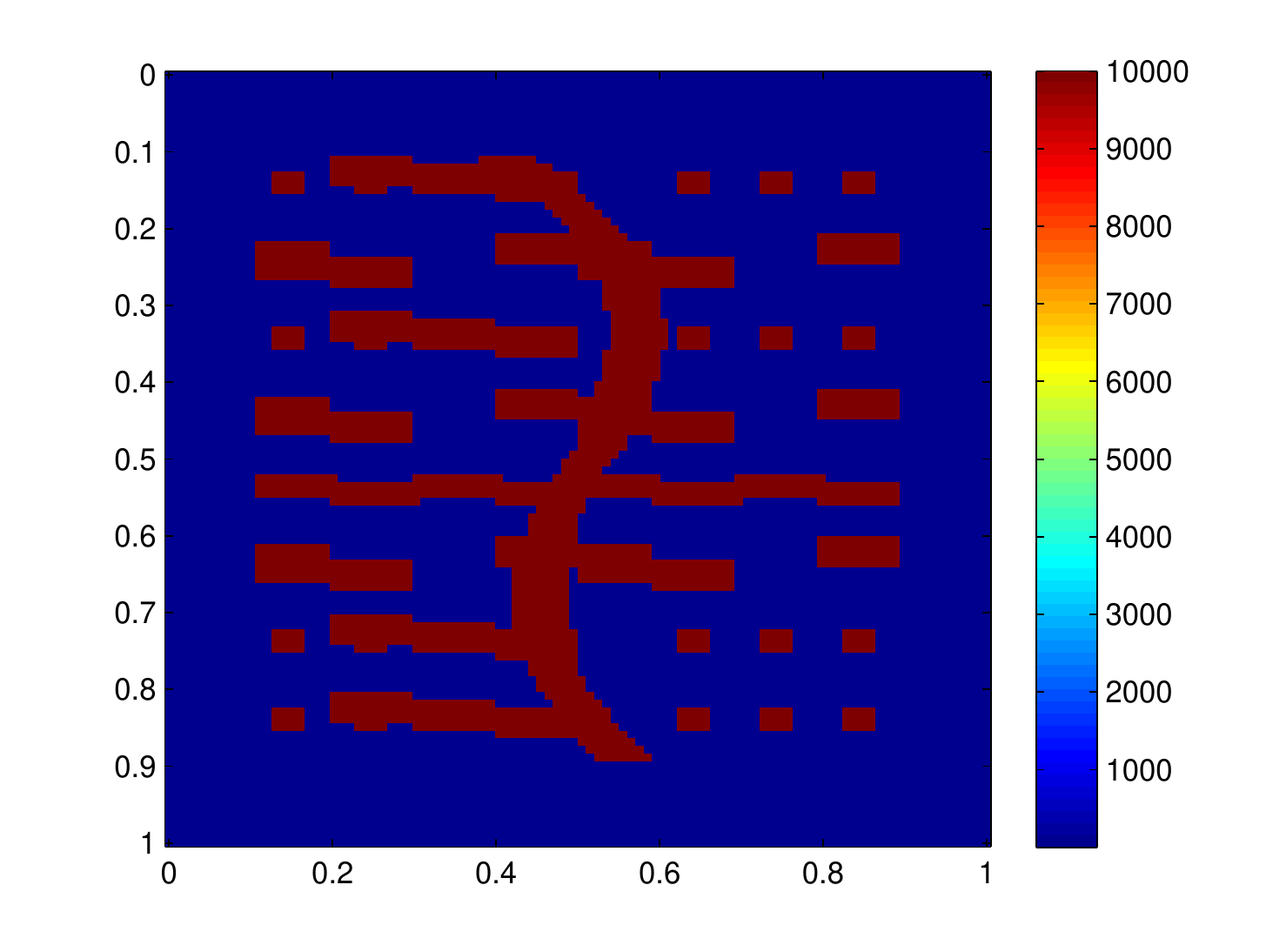}}		
	\caption{Heterogeneous test models.}
	\label{fig:model}
\end{figure}

\begin{figure}[H]
	\centering
	\subfigure[Approximation comparison for $g_1$]{
		\includegraphics[width=2.8in,height=2.1in]{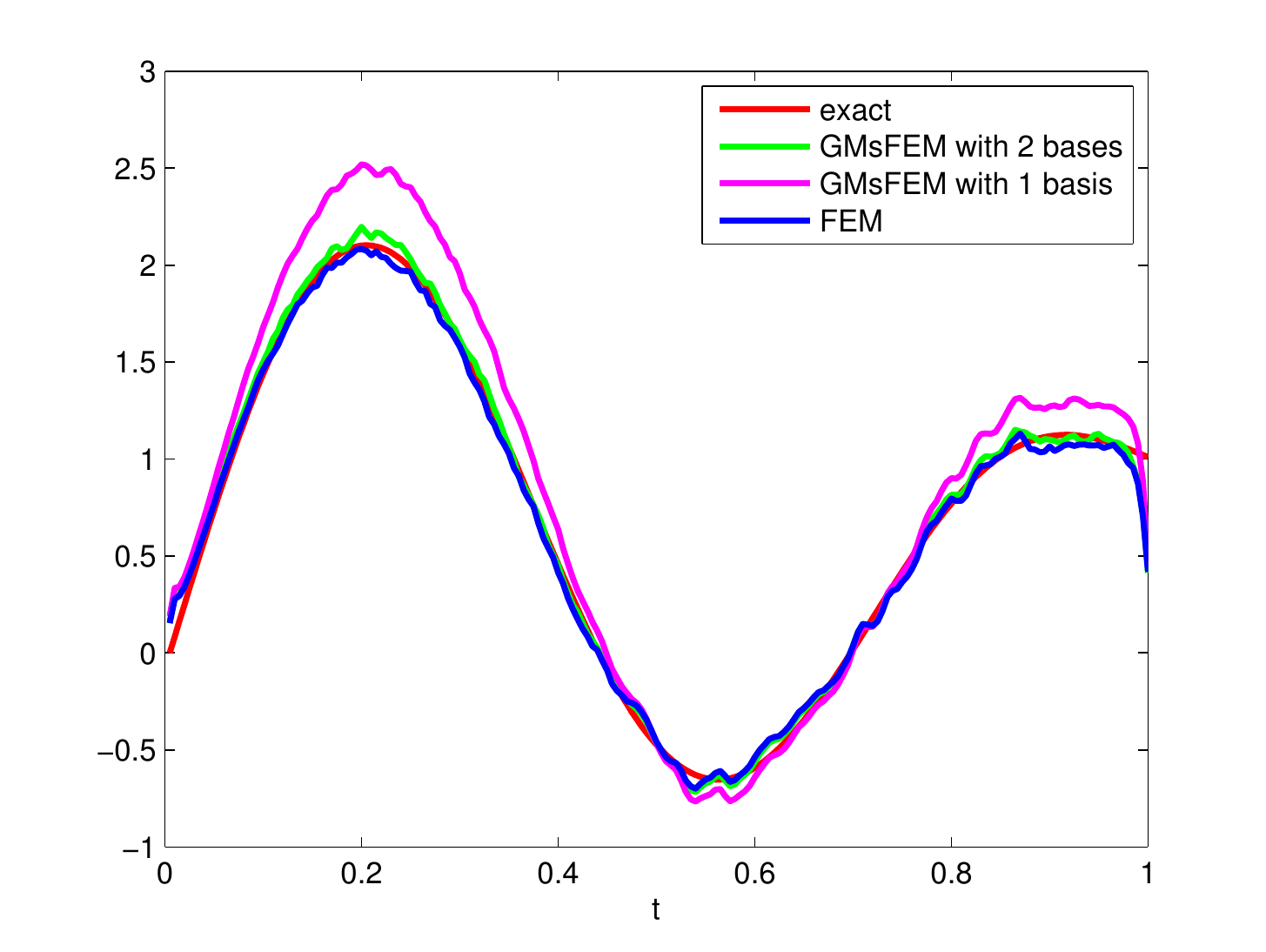}}	
	\subfigure[Approximation comparison for $|g_2|$]{
		\includegraphics[width=2.8in,height=2.1in]{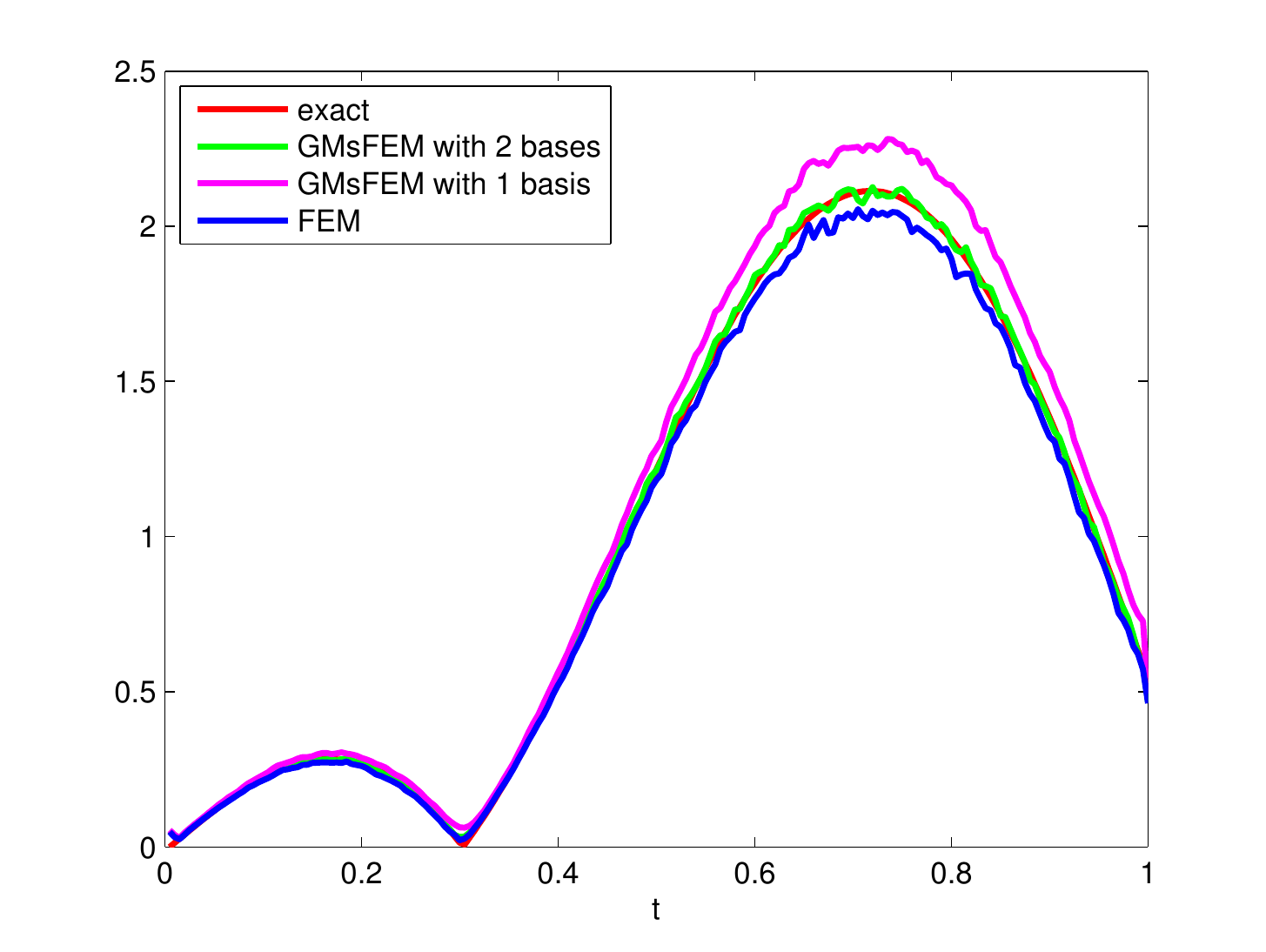}}
	\caption{Results for test model 2, smooth case.}
	\label{fig:het}
\end{figure}

\begin{figure}[H]
	\centering
	\subfigure[Approximation comparison for $g_1$]{
		\includegraphics[width=2.8in,height=2.1in]{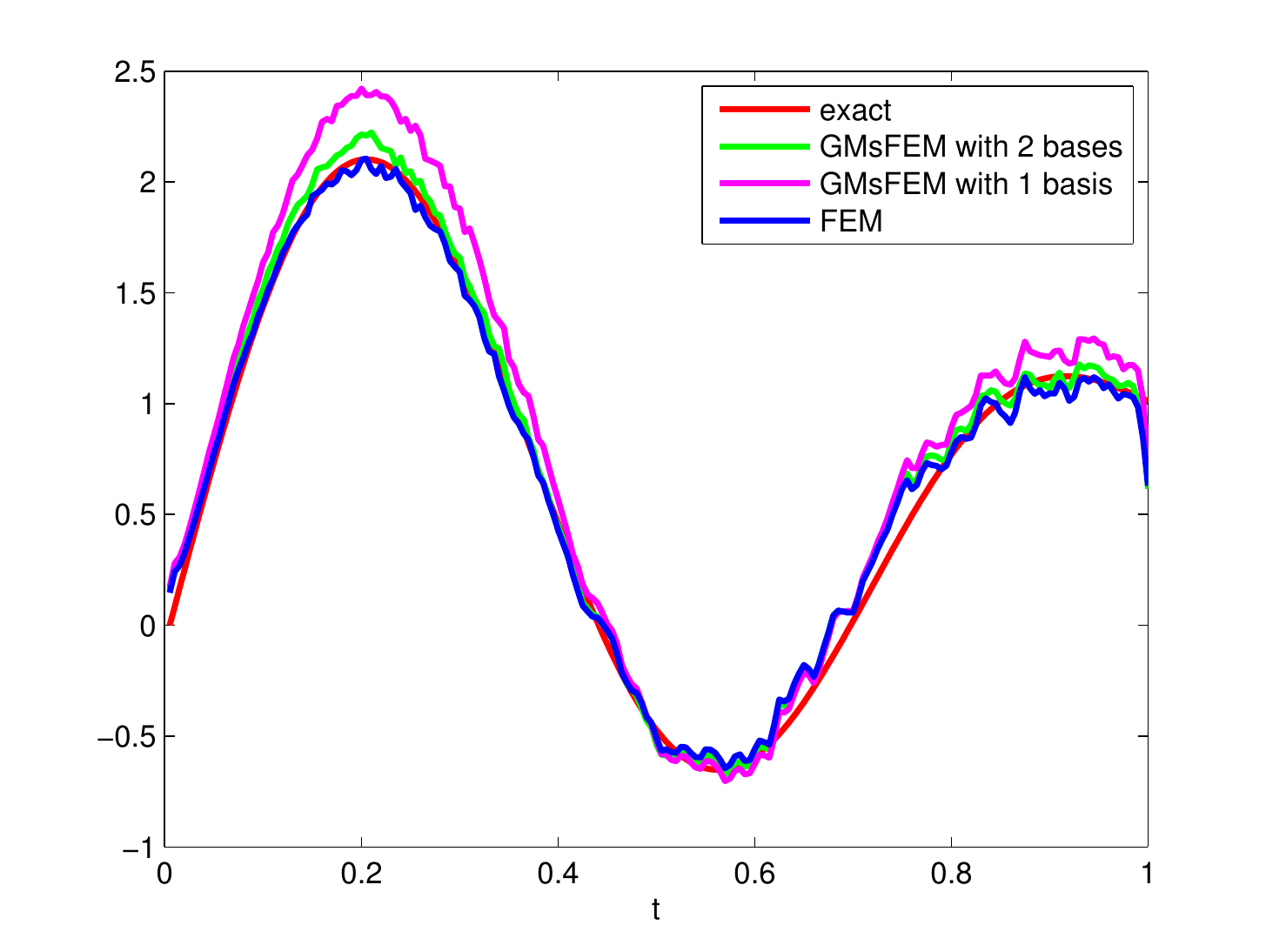}}	
	\subfigure[Approximation comparison for $|g_2|$]{
		\includegraphics[width=2.8in,height=2.1in]{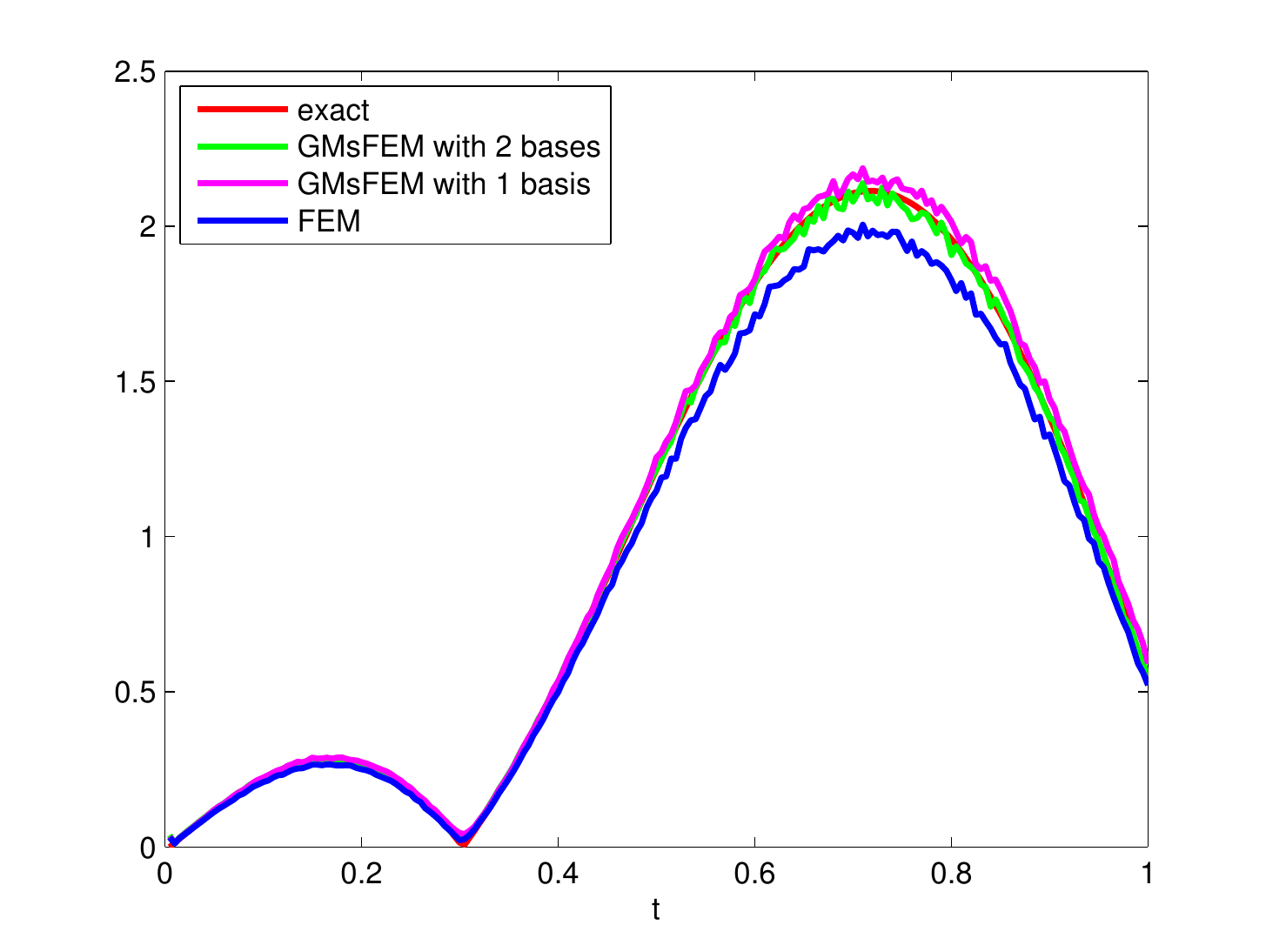}}
	\caption{Results for test model 3, smooth case.}
	\label{fig:cross}
\end{figure}

\begin{figure}[H]
	\centering
	\subfigure[Approximation comparison for $g_1$]{
		\includegraphics[width=2.8in,height=2.1in]{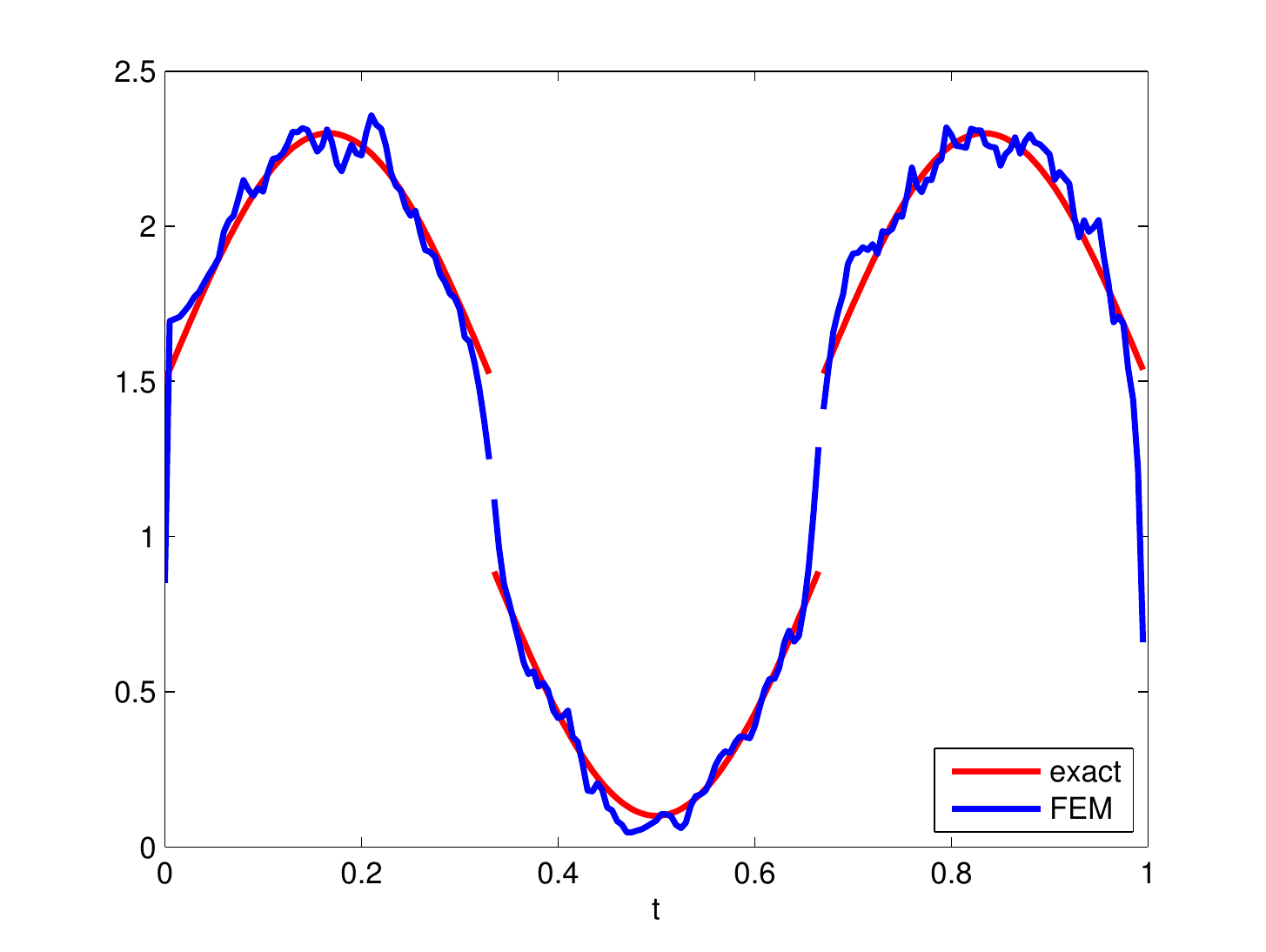}}	
	\subfigure[Approximation comparison for $|g_2|$]{
		\includegraphics[width=2.8in,height=2.1in]{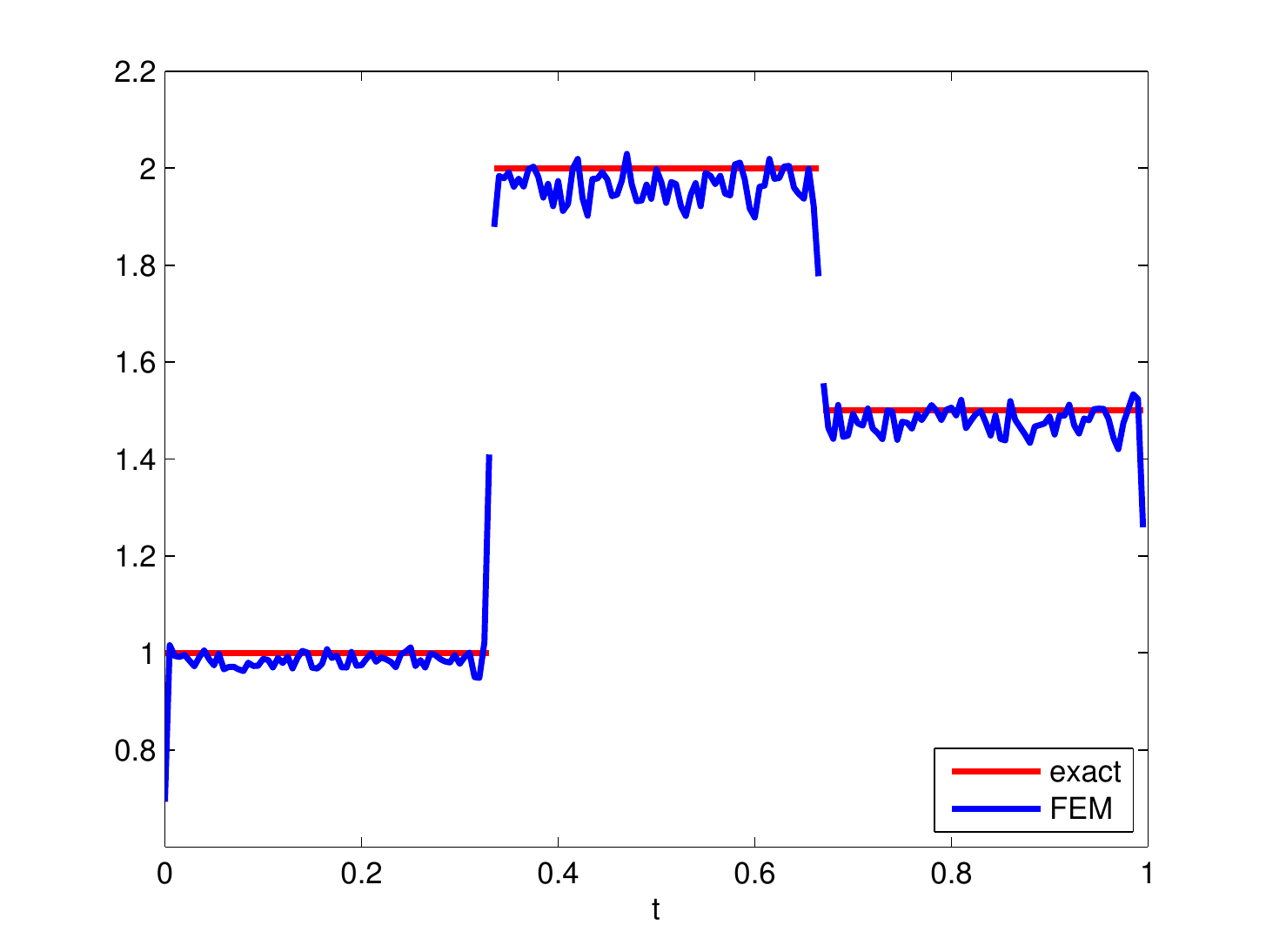}}
	\caption{Results for homogeneous model (test model 1), non-smooth  case.}
	\label{fig:homo_dis}
\end{figure}
\begin{figure}[H]
	\centering
	\subfigure[Approximation comparison for $g_1$]{
		\includegraphics[width=2.8in,height=2.1in]{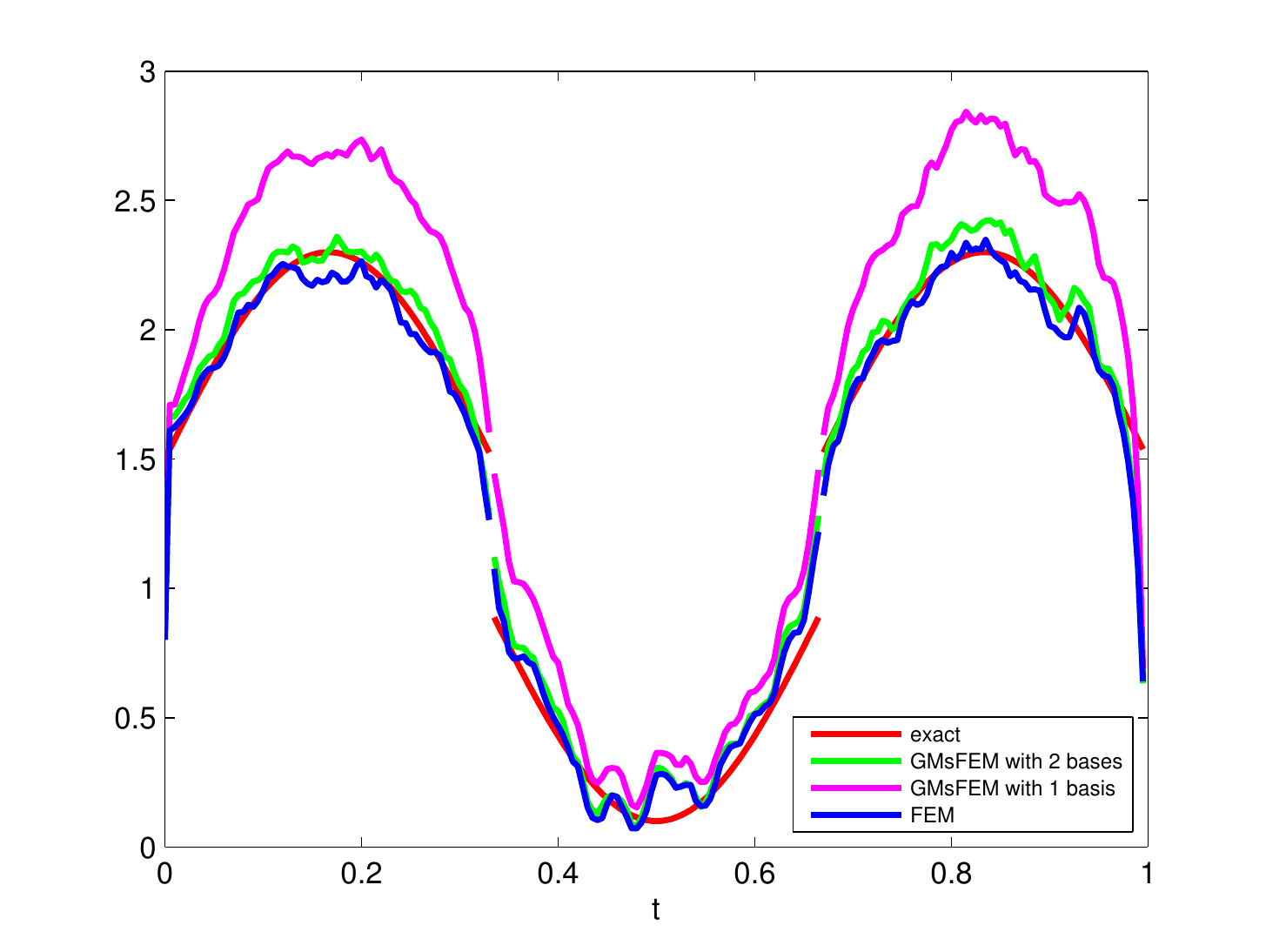}}	
	\subfigure[Approximation comparison for $|g_2|$]{
		\includegraphics[width=2.8in,height=2.1in]{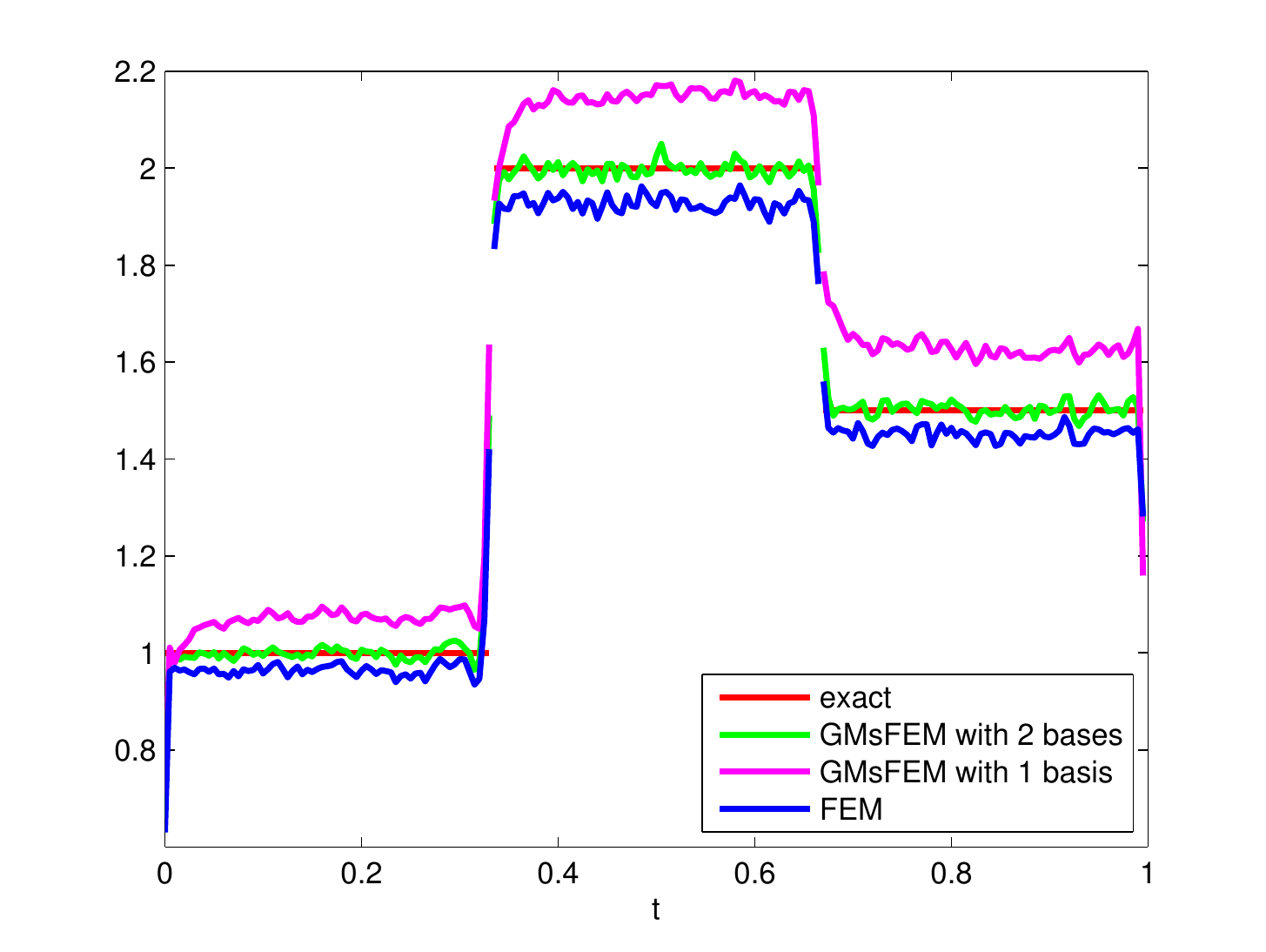}}
	\caption{Results for test model 2, non-smooth case.}
	\label{fig:het_dis}
\end{figure}

\begin{figure}[H]
	\centering
	\subfigure[Approximation comparison for $g_1$]{
		\includegraphics[width=2.8in,height=2.1in]{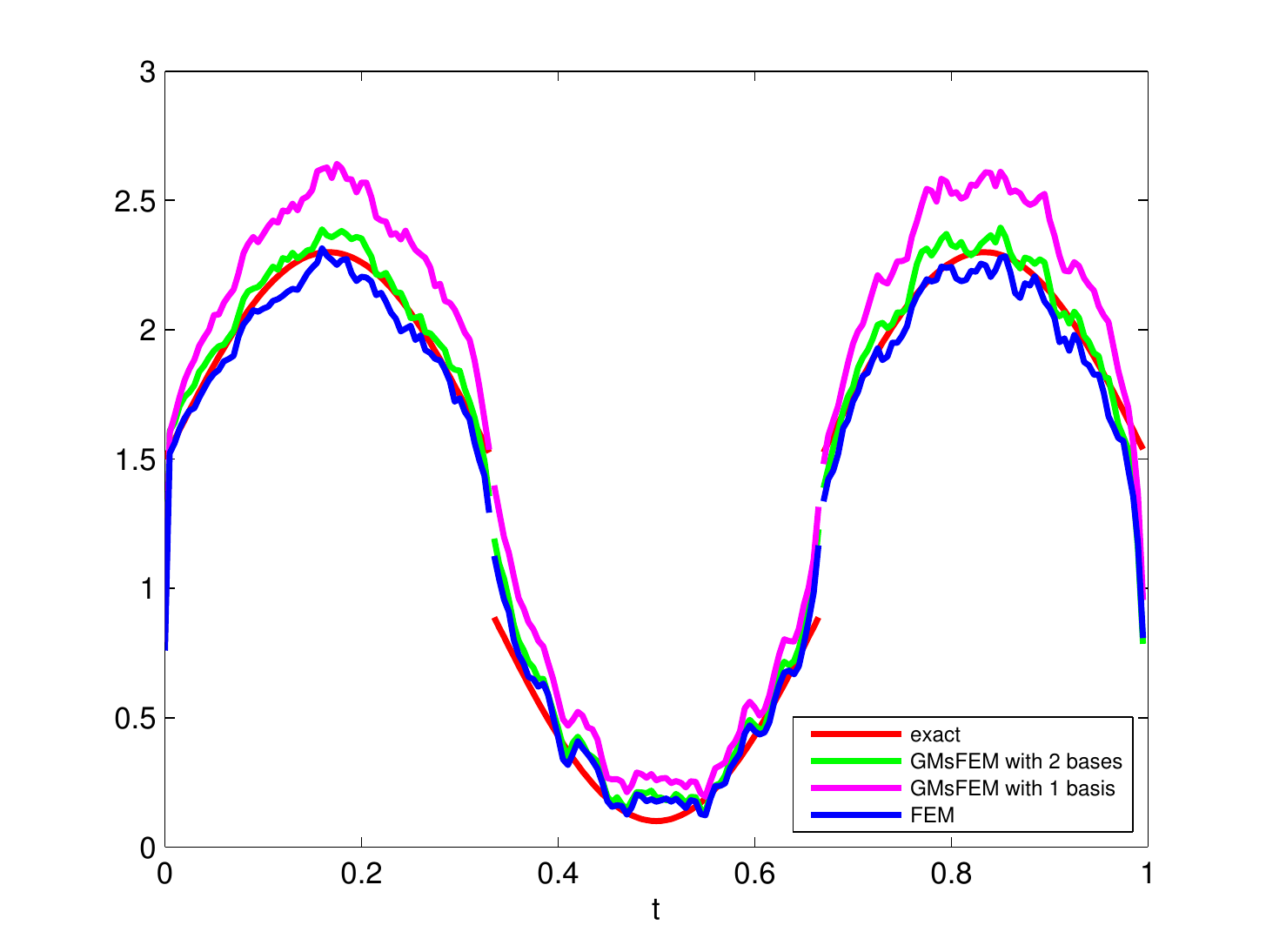}}	
	\subfigure[Approximation comparison for $|g_2|$]{
		\includegraphics[width=2.8in,height=2.1in]{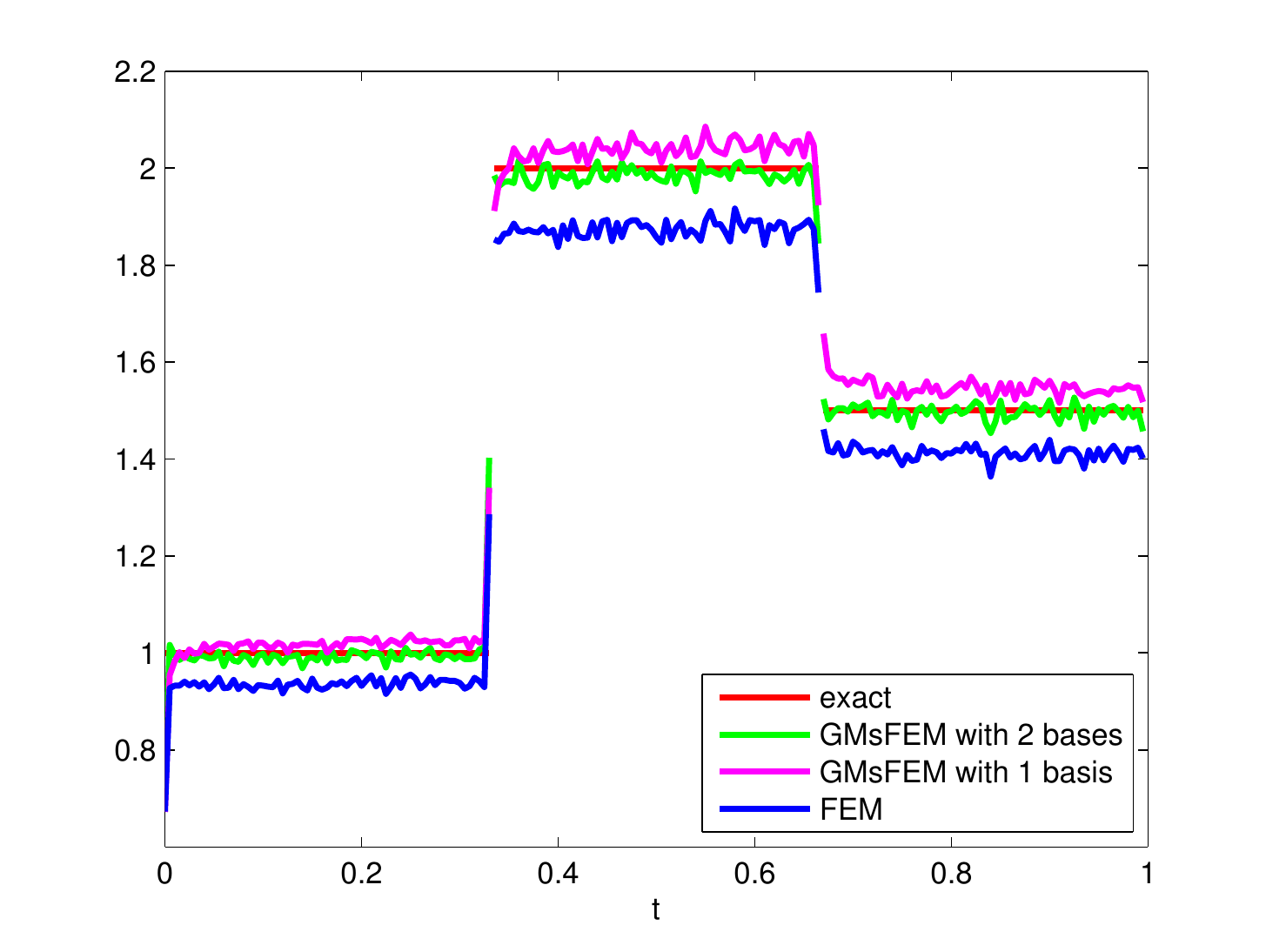}}
	\caption{Results for test model 3, non-smooth case.}
	\label{fig:cross_dis}
\end{figure}

Furthermore, the non-smooth case is also tested, 
\begin{equation*}
\begin{aligned}
g_1(t)=& (1.5+0.8\sin(3\pi t))\chi_{{}_{ t\in[0,1/3)\cup[2/3,1]}}
+(0.9+0.8\sin(3\pi t))\chi_{{}_{t\in[1/3,2/3)}},\\
g_2(t)=&\chi_{{}_{ t\in[0,1/3)}}-2\chi_{{}_{t\in[1/3,2/3)}}
+1.5\chi_{{}_{t\in[2/3,1]}}.
\end{aligned}
\end{equation*}
We present the corresponding results in Figures \ref{fig:homo_dis}-\ref{fig:cross_dis}. Again, we observe that our algorithm works well, and the GMsFEM with 2 bases is better than the FEM especially for the 
approximation of $|g_2|$, which agrees with the smooth case.

\section{Concluding remark and future work} 
This paper considers the recovery of unknown source in the stochastic fractional diffusion equation. The statistical moments of single point data $u(x_0,t,\o)$ are used, and the observation point $x_0$ is set to be out of the support of the source, which fits the practical circumstance. The restriction on $x_0$ makes the analysis more challenging. Nonetheless, the estimates of unknowns on the incomplete interval are given, and the constructed iterative algorithm works for both smooth and non-smooth cases.  

From the numerical results, one natural question for this inverse problem will be whether we can recover more information about the unknown $g_2(t)$. In this work, we can only solve $g_2^2$, or $|g_2|$, which can not describe        
$g_2$ well. This comes from the Ito formula in Lemma \ref{Ito isometry formula}, by which the term $g_2^2$ is generated. Hence, if we want to reconstruct $g_2$ further, some more complicated statistical moments need to be used, not only variance. For example, it seems that we may obtain $\pm g_2$ from the moment covariance. The corresponding investigation is one of our future work.      

\section*{Acknowledgment}
The second author was supported by Academy of Finland, grants 284715, 312110, and the Atmospheric mathematics project of University of Helsinki.

\bibliographystyle{abbrvurl} 
\bibliography{SFDE_time}

\end{document}